\documentclass[12pt]{amsart}%
\usepackage{amssymb}
\usepackage{amsfonts}
\usepackage{amsmath}
\usepackage{color}
\usepackage{graphicx}%
\usepackage{comment}
\providecommand{\U}[1]{\protect\rule{.1in}{.1in}}
\newtheorem{theorem}{Theorem}

\newtheorem{definition}[theorem]{Definition}

\newtheorem{lemma}[theorem]{Lemma}

\newtheorem{proposition}[theorem]{Proposition}
\newtheorem{remark}[theorem]{Remark}

\numberwithin{equation}{section}
\newcommand{\vep}{\varepsilon}
\renewcommand{\d}{\mathrm{d}}

\definecolor{ggreen}{rgb}{0.55,0.71,0.0}

\begin{document}

\title[Nonpotential perturbations to doubly-nonlinear flows]{Elliptic-regularization of nonpotential
  perturbations of doubly-nonlinear gradient flows of nonconvex energies: A variational approach}

\thanks{G.A.~is supported by JSPS KAKENHI Grant Number 16H03946 and by the Alexander von Humboldt Foundation and by the Carl Friedrich von Siemens Foundation. S.M.~is supported by the Austrian Science Fund (FWF) project P27052-N25.  The Authors would like to acknowledge the kind hospitality of the Erwin Schr\"odinger International Institute for Mathematics and Physics, where part of this research was developed under the frame of the Thematic Program {\it Nonlinear Flows}.
}

\author{Goro Akagi}
\address[ Goro Akagi]{Mathematical Institute, Tohoku University, Aoba, Sendai 980-8578 Japan; Helmholtz Zentrum M\"unchen, Institut f\"ur Computational Biology, Ingolst\"adter Landstra\ss e 1, 85764 Neunerberg, Germany; Technische  Universit\"at  M\"unchen, Zentrum Mathematik, Boltzmannstra\ss e 3, D-85748 Garching bei M\"unchen, Germany.}
\email{akagi@m.tohoku.ac.jp}
\author{Stefano Melchionna}
\address[Stefano Melchionna]{ University of Vienna, Faculty of Mathematics,
Oskar-Morgenstern-Platz 1, 1090 Wien, Austria.}
\email{stefano.melchionna@univie.ac.at}
\maketitle

 \begin{abstract}
This paper presents a variational approach to  doubly-nonlinear  (gradient) flows  (P)  of nonconvex energies along with  nonpotential  perturbations (i.e., perturbation terms without any potential structures). An elliptic-in-time regularization of the original equation  ${\rm (P)}_\vep$  is introduced, and then, a variational approach and a fixed-point argument are employed to prove existence of strong solutions to regularized equations. More precisely, we introduce a functional (defined for each entire trajectory and including a small approximation parameter $\varepsilon$) whose Euler-Lagrange equation corresponds to the elliptic-in-time regularization of an unperturbed (i.e.~without  nonpotential  perturbations)  doubly-nonlinear  flow. Secondly, due to the presence of  nonpotential  perturbation, a fixed-point argument is performed to construct strong solutions $u_\varepsilon$ to the elliptic-in-time regularized equations  ${\rm (P)}_\vep$. Here, the minimization problem mentioned above defines an operator $S$ whose fixed point corresponds to a solution $u_\varepsilon$ of  ${\rm (P)}_\vep$. Finally, a strong solution to the original equation  (P)  is obtained by passing to the limit of $u_\vep$ as $\varepsilon \to 0$. Applications of the abstract theory developed in the present paper to concrete PDEs are also exhibited.
 \end{abstract}


\section{Introduction}

In this paper, we deal with a \textit{ nonpotential  perturbation} problem ~(P) = $\{\eqref{target equation},\eqref{IC}\}$ of a  doubly-nonlinear  (gradient) flow driven by a \emph{dissipation potential} $\psi$ and a (possibly) nonconvex \emph{energy functional} $\phi$ defined on a uniformly convex Banach space $V$,
\begin{align}
\mathrm{d}_{V}\psi(u^{\prime})+\partial\phi(u)-f(u)  &  \ni0  \ \text{ a.e. in
}(0,T)\text{,}\label{target equation}\\
u(0)  &  =u_{0} \label{IC}%
\end{align}
where $u^{\prime}$ denotes the time derivative of the unknown $u:[0,T]\rightarrow V$, $\psi:V\rightarrow\lbrack0,\infty]$ is supposed to be convex and G\^ateaux differentiable and its (G\^ateaux) derivative is denoted by $\d_V \psi$, $\partial \phi$ is a derivative of $\phi$ (in a proper sense) and $f(\cdot)$ is a  nonpotential  mapping from $V$ into its dual space $V^*$ (see below for more details). We assume that $\phi$ can be decomposed into the difference of two convex functionals,
\[
\phi=\varphi^{1}-\varphi^{2}
\]
where $\varphi^{1},\varphi^{2}:V\rightarrow(-\infty,\infty]$ are proper, lower semicontinuous, and convex functionals.  We assume that $\varphi_1$ dominates $\varphi_2$ in a suitable sense (cf. (\ref{phi2 bound})). This ensures that the difference
$\phi$ is well defined.  Here and henceforth, we simply write $\partial \phi=\partial\varphi^{1}-\partial\varphi^{2}$, where the symbol $\partial$ in the right-hand side denotes the subdifferential of convex analysis, unless any confusion may arise. Let us emphasize that we do not assume any potential structure on $f$ (e.g. $f=\partial F$), although we shall impose the continuity as well as some growth condition on the map $f$. Hence, throughout the paper, the perturbation term $f$ is said to be \textit{ nonpotential }.

The study of  doubly-nonlinear  evolution equations of the form \eqref{target equation} with $f \equiv 0$ was initiated by Barbu~\cite{Barbu75}, Arai~\cite{Arai}, Senba~\cite{Senba}, Colli-Visintin~\cite{CV}, and  Colli~\cite{Colli}, and then, they have been vigorously studied so far by many authors (see, e.g.,~\cite{AiziYan},~\cite{Ak},~\cite{GO3},~\cite{ASc},~\cite{Ak-St},~\cite{BoSc04},~\cite{Fr1,Fr2},~\cite{RR15},~\cite{MiTh},~\cite{MiRo},~\cite{MRS},~\cite{Roubicek},~\cite{SSS},
~\cite{Segatti},~\cite{S-ken},~\cite{St08},\cite{Visintin}, and references therein). Many of them are motivated in view of applications to physics and engineering, for  doubly-nonlinear  equations are often introduced to describe important irreversible phenomena such as phase transition, friction, damage, and so on. Equation \eqref{target equation} is extremely general and can cover an extensive class of nonlinear PDEs and evolution equations appeared in the field of dissipative phenomena (e.g., nonlinear diffusion and phase transition). The formulation (\ref{target equation}) covers standard gradient  flows  (i.e., the case $\d_V \psi(u') = u'$ which corresponds to a quadratic dissipation potential $\psi$ on a Hilbert space),  doubly-nonlinear  flows (for general dissipation potentials defined on Banach spaces), and moreover, it can also comply with  nonpotential  perturbations. Particularly, in the study of PDEs, one can often find out (possibly partial) gradient structures concealed in many PDEs describing dissipation phenomena, and such a gradient structure induces a leading feature of each equation; however, in most of cases, equations do not have full gradient structure and they also include some reminder terms which prevent us to reduce the equations into complete forms of gradient flows (e.g.,  typical examples may be Navier-Stokes equation and systems of PDEs.  See also \^Otani~\cite{Ot1,Ot2}). In order to cover such a wider class of equations (with partial gradient structures), we shall develop a  nonpotential  perturbation theory for  doubly-nonlinear  flows. In Section \ref{applications}, we shall treat some system of PDEs as an example of such equations (see also \cite{Me}). For the quadratic dissipation in the Hilbert space setting, such perturbation problems have been intensively studied by many authors (e.g,~\cite{YO},~\cite{Ot1,Ot2}, and~\cite{SIYK98}).

The variational approach which we shall apply to \eqref{target equation} is based on the so-called \textit{Weighted Energy-Dissipation} (WED) functional. This approach consists of introducing a one-parameter family of WED functionals $I_\vep$ defined over entire trajectories, and proving that their minimizers converge, up to  subsequences, to \emph{strong}  solutions  of the target problem, as the approximation parameter $\vep$  approaches to zero. Our interest in the WED formalism lies in the fact that it  paves the way  to the application of general techniques of the calculus of variation (e.g., Direct Method, relaxation, the $\Gamma$-convergence) in the evolutionary setting. Moreover, the WED procedure also brings a new tool to reveal qualitative properties of solutions and comparison principles for evolutionary problems  \cite{Me2}. Furthermore, also in the present paper, this variational formulation brings us a useful technique to check the uniqueness of solutions and structural stability of unperturbed equations from the strict convexity of the WED functionals and $\Gamma$-convergence theory, respectively. Indeed, uniqueness of WED minimizer is used to define a (single-valued) \emph{solution operator}, which maps a prescribed function $v$ to the corresponding solution $u$ (cf.~for  (not regularized) doubly-nonlinear  flows, uniqueness of solution is delicate; indeed, in some cases, it is false (see, e.g.,~\cite{Colli} and~\cite{A-PJ})), and the structural stability implies the continuity of the solution operator, which will be required to apply a  fixed-point  theorem (see Section \ref{convex case} below for more details). Furthermore, the minimization problem provides more regular (in time) solutions, for the Euler-Lagrange equation associated with the WED functional corresponds to an elliptic-in-time regularization of the target problem. The elliptic-regularization approach to evolution equations has to be traced back at least to~\cite{Li} and~\cite{Ol} (see also~\cite{Li-Ma}). The idea of the WED functional approach has already been used in~\cite{Il} and~\cite{Hi}. Later, it has been reconsidered by Mielke and Ortiz~\cite{Mi-Or} for rate-independent equations, by Mielke and Stefanelli~\cite{Mi-St} for gradient flows with $\lambda$-convex potentials, and by Akagi and Stefanelli for non($\lambda$-) convex gradient flows~\cite{Ak-St} and  doubly-nonlinear  problems~\cite{AMS, Ak-St1,Ak-St2,Ak-St3}. Finally, the WED approach to  nonpotential  perturbations of gradient flows of nonconvex energies was recently developed in~\cite{Me}.

Elliptic-in-time regularization is one of well-established methods and has been widely employed in various fields including numerical analysis and control theory. Indeed, it provides us more regular approximation of solutions and more choices of methods to tackle the target equation (for instance, methods for elliptic PDEs are also available for parabolic and hyperbolic PDEs). On the other hand, for severely nonlinear evolution equations, there still remain many fundamental open issues such as existence of \emph{strong} (i.e., twice differentiable) solutions for elliptic-in-time regularized equations and the convergence of such approximate solutions to a solution of the target equation. Indeed, applications of elliptic-in-time regularization are based on these fundamental hypotheses, and they should be mathematically justified for each equation. One of main purposes of this paper is to propose a general theory to guarantee the fundamental hypotheses of elliptic-in-time regularization for a wider class of dissipative evolution equations as well as to extend the theory of~\cite{Ak-St},~\cite{Ak-St2},~\cite{Ak-St3} and~\cite{Me} to  nonpotential  perturbation problems for  doubly-nonlinear  flows of nonconvex energies.

The main result of this paper ensures that a solution to (\ref{target equation})-(\ref{IC}) is obtained as the limit of solutions $u_\vep$ to an elliptic-in-time regularization ${\rm (P)}_\vep$ given by
\begin{gather*}
-\varepsilon \dfrac \d {\d t} \mathrm{d}_{V}\psi(u^{\prime}) + \mathrm{d}_{V}\psi(u^{\prime}) + \partial \phi(u)-f(u)\ni 0 \quad \text{ a.e. in
}(0,T)\text{,}\\
u(0)=u_{0}\text{, } \quad \mathrm{d}_{V}\psi(u^{\prime}(T))=0%
\end{gather*}
as $\vep \to 0$ (to be precise, we shall treat a \emph{weak formulation} of ${\rm (P)}_\vep$. See Definition \ref{def sol eu} below). Problems ${\rm (P)}_\vep$ will be tackled by combining the minimization of WED functionals and a  fixed-point  argument. In particular, if $f(u)$ is replaced by $w := f(v)$ with a prescribed function $v$,  we prove existence of solutions to the corresponding problem by minimization of a suitably defined WED functional (cf.~Section \ref{section wed minimization}). We also note that, although uniqueness for the regularized unperturbed problem is not known, the minimization problem features uniqueness of solutions due to the strict convexity of the WED functional, and moreover, it also guarantees a continuous dependence of the solution $u = u(t)$ on the prescribed function $v = v(t)$.

The paper is organized as follows. In Section \ref{assumptions}, we set up notation, enlist our assumptions and state our main results. In order to simplify the argument, we first prove our results only for convex energy functionals (namely, the case $\varphi^{2}=0$) in Section \ref{convex case}. Secondly, we extend the result of  Section  \ref{convex case} to general nonconvex energies in Section \ref{nonconvex energies}. In Section \ref{section wed minimization}, as a by-product, we shall develop
a variational characterization based on WED functionals for  doubly-nonlinear  flows of nonconvex energies (i.e., (P) with $f=0$). In Section \ref{second approach}, we briefly sketch a second fixed-point argument, which allows us to work under
slightly different assumptions on the  nonpotential  term $f$. Finally, Section
\ref{applications} concerns some applications of the preceding abstract theory to concrete PDEs.

\section{Main results\label{assumptions}}

Let $V$ be a uniformly convex Banach space with norm $|\cdot|_{V}$ and duality
pairing $\left\langle \cdot,\cdot\right\rangle _{V}$, and  let $\left(
V^{\ast},|\cdot|_{V^{\ast}}\right)$ be its dual space such  that $V^{\ast}$ is also uniformly convex. Here it would be noteworthy that we do not assume the presence of a pivot (Hilbert) space $H \equiv H^*$ between $V$ and $V^*$ due to our variational approach. The pivot space $H$ is often required, if \eqref{target equation} is treated by virtue of approximation techniques (in $H$) such as Yosida approximation. Moreover, in applications to nonlinear PDEs, such an abstract framework based on the Gel'fand triplet $V \hookrightarrow H \equiv H^* \hookrightarrow V^*$ may impose additional assumptions on the PDEs (cf. see~\cite{Ak} and also Remark \ref{R:triplet}). 

Let $X$ be a reflexive Banach space with norm $|\cdot|_{X}$ and
duality pairing $\left\langle \cdot,\cdot\right\rangle _{X}$. Suppose that%
\[
X\hookrightarrow V\text{ and }V^{\ast}\hookrightarrow X^{\ast}%
\]
with compact densely-defined canonical injections. Let $\psi:V\rightarrow
\lbrack0,\infty)$ be a G\^{a}teaux differentiable convex functional and  let $\varphi^{1},\varphi^{2}:V\rightarrow \lbrack0,\infty]$ be two proper, lower semicontinuous, and convex functionals whose \emph{effective domains} (i.e., the set of $u$ for which $\varphi^i(u) < +\infty$) are  denoted  by $D(\varphi^{1})$ and $D(\varphi^{2})$, respectively. Let $p\in
(1,\infty)$ and $m\in(1,\infty)$ be fixed and assume the following:

\begin{description}
\item[(A1)] There exists $C_{1}>0$ such that
\begin{equation}
|u|_{V}^{p}\leq C_{1}(\psi(u)+1) \label{coercivity psi}%
\end{equation}
$\ $\ for all $u\in V$;
\item[(A2)] There exists $C_{2}>0$ such that
\begin{equation}
|\mathrm{d}_{V}\psi(u)|_{V^{\ast}}^{p^{\prime}}\leq C_{2}(|u|_{V}^{p}+1)
\label{p growth psi 2}
\end{equation}
for all $u\in V$, where $\mathrm{d}_{V}\psi:V \to V^*$ denotes the G\^{a}teaux
derivative of $\psi$;
\item[(A3)] There exists a positive and nondecreasing function $\ell_{3}$ on $[0,\infty)$ such that
\begin{equation}
|u|_{X}^{m}\leq \ell_{3}(|u|_V) (\varphi^{1}(u)+1) \label{coercivity phi}%
\end{equation}
for all $u\in D(\varphi^{1});$

\item[(A4)] Let $\varphi_{X}^{1}:X\rightarrow\lbrack0,\infty)$ be the restriction of $\varphi^{1} : V \to [0,\infty]$ onto $X$ and let $\partial_{X}\varphi_{X}^{1} : X \to X^*$ be its subdifferential. There exists a  nondecreasing  function $\ell_{4}(\cdot)$ on  $[0,\infty)$ such that%
\[
|\eta^{1}|_{X^{\ast}}^{m^{\prime}}\leq \ell_4(|u|_V) (|u|_{X}^{m}+1)
\]
for all  $u \in D(\varphi_X^1 )$, $\eta^{1} \in\partial_{X}\varphi_{X}^{1}(u)$;

\item[(A5)] Let $f:V\rightarrow V^{\ast}$ be  continuous and such that
\begin{equation}
|f(u)|_{V^{\ast}}^{p^{\prime}}\leq C_{5}\left(  |u|_{V}^{p}+1\right)
\label{growth f}%
\end{equation}
for all $u \in V$ and for some positive constant $C_{5}$ (independent of $u$). 

\item[(A6)] $D(\varphi^{1}) \subset D(\partial_{V}\varphi^{2})$. Moreover, there
exist\ constants $k\in\lbrack0,1),~C_{6}>0$, and a nondecreasing function
 $\ell_7$ on $[0,\infty)$  such that
\begin{equation}
\varphi^{2}(u)\leq k\varphi^{1}(u)+C_{6} \left( |u|_V^p + 1 \right) \label{phi2 bound}%
\end{equation}
for all $u\in D(\varphi^{1})$, and
\begin{equation}
|\eta^{2}|_{V^{\ast}}^{p^{\prime}}\leq \ell_7 (|u|_{V})(\varphi^{1}(u)+1)
\label{cond}%
\end{equation}
for all $u\in D(\varphi^{1}), \eta^{2}\in\partial_{V}\varphi
^{2}(u)$.  We now define the energy potential $\phi:V\to (-\infty,\infty]$ by 
$\phi(u)=\varphi^{1}(u)-\varphi^{2}(u)$ for all $u\in D(\varphi_1)$ and $\phi(u)=\infty$ for all $u\in V\setminus D(\varphi_1)$.   
\end{description}

As a consequence of \textbf{(A1)-(A4)} there exist constants  $C_{i}$,
$ i \in \{8,9\}$ and nondecreasing functions $\ell_{10}(\cdot)$, $\ell_{11}(\cdot)$  in $\mathbb{R}$ such that
\begin{align}
|u|_{V}^{p}  &  \leq  C_{8} \left(  \left\langle \d_{V}\psi(u),u\right\rangle
_{V}+1\right)  \text{ for all } u\in V,\label{assumptions consequence}\\
\psi(u)  &  \leq  C_{9} \left(  |u|_{V}^{p}+1\right)  \text{ for all } u\in V,\\
|u|^{m}_{ X }  &  \leq  \ell_{10}(|u|_V) ( \left\langle \eta^{1},u\right\rangle _{X}+1) \text{ for all } u \in D(\partial_{X}\varphi_{X}^{1}),\eta^{1} \in\partial_{X}\varphi_{X}^{1}(u)\text{,}\label{A3'}\\
 \varphi^1 (u)  &  \leq  \ell_{11}(|u|_V) \left(  |u|_{X}^{m}+1\right)  \text{ for all } u\in
D(\varphi^{1})\text{.}%
\end{align}

Finally, we assume
\begin{equation}
u_{0}\in D(\varphi^{1}). \label{ic}%
\end{equation}

 \begin{remark} \label{remark assumptions f} {\rm
  \begin{enumerate}
\item[(i)] Under \textbf{(A3)} and \textbf{(A4)}, one can check $D(\varphi^1) = D(\varphi^1|_X) = X$. Hence due to~\cite[Proposition 2.2]{Ba}, $\varphi^1|_X$ turns out to be continuous from $X$ to $[0,\infty)$ (see Lemma \ref{L:X} in Appendix). Hence \eqref{ic} is equivalent to $u_0 \in X$.
   \item[(ii)] By \textbf{(A5)}, the mapping $u \mapsto f(u)$ from $L^p(0,T;V)$ into $L^{p'}(0,T;V^*)$ turns out to be continuous (see, e.g.,~\cite{Roubicek}).
   \item[(iii)] Main results of this paper can be extended to more general perturbations: indeed, we may assume instead of \textbf{(A5)} that $f$ satisfies \eqref{growth f} and $f$ is demicontinuous (i.e. strongly-weakly continuous) from $L^{p}(0,T;V)$ to $L^{p^{\prime}}(0,T;V^{\ast})$, i.e.,
\begin{equation}\label{f demicon}
 \left\{
  \begin{array}{l}
  \mbox{if }
   u_{n}\rightarrow u\text{ strongly in }L^{p}(0,T;V),\\
   \mbox{then } f(u_{n}) \to f(u)\text{ weakly in }L^{p^{\prime}}(0,T;V^{\ast})\text{,}
\end{array}
 \right.
\end{equation}
	       which is equivalent to the demiclosedness under \eqref{growth f},
and the following compactness condition of the mapping $f$ from $W^{1,p}(0,T;V)\cap L^{m}(0,T;X)$ to $L^{p^{\prime}}(0,T;V^{\ast})$:%
\begin{equation}
 \left\{
  \begin{array}{l}
  \text{if } (u_{n}) \text{ is bounded in } L^{m}(0,T;X) \text{ and in }W^{1,p}(0,T;V), \\
   \text{and } (f(u_{n})) \text{ is bounded in }L^{p^{\prime}}(0,T;V^{\ast })\text,\\
   \text{then there exists } f^{\ast} \in L^{p'}(0,T;V^*) \text{ such that, } \text{up to a subsequence, }\\
    f(u_{n})\rightarrow f^{\ast}\text{ strongly in }L^{p^{\prime}}(0,T;V^{\ast}).
  \end{array}
 \right.  \label{hemicontinuity f}%
\end{equation}
	       By \eqref{f demicon} as well as the Aubin-Lions-Simon compactness lemma, the limit $f^*$ is identified with $f(u)$.
\item[(iv)] Note that (constant-in-time) external forces can be considered by choosing the term $f$ to be independent of the variable $u$. We remark that our results can be derived with no major essential differences also for problems with time dependent external forces. More precisely, we can relax assumption \textbf{(A5)} by requiring that $f(u)=\tilde{f}(u)+g$, where $\tilde{f}$ satisfies \textbf{(A5)} and $g:(0,T)\to V^\ast$ belongs to the space $L^{p^\prime}(0,T;V^\ast)$. We refer the reader to \cite{Ak-St3} for the WED approach to (unperturbed) doubly-nonlinear systems with time-dependent external forces.
  \end{enumerate}
 }\end{remark}

Now, we are concerned with the Cauchy problem,
\begin{gather}
\mathrm{d}_{V}\psi(u^{\prime})+\eta^{1}-\eta^{2}-f(u)  =0 \  \mbox{ in } V^*  \text{ a.e. in }(0,T)\text{,}\label{nonconvex target equation}\\
\eta^{1} \in\partial _{ V } \varphi^{1}(u)\text{,} \quad \eta^{2} \in\partial_{ V } \varphi
^{2}(u),\\
u(0)   =u_{0}\text{.} \label{ic nc}%
\end{gather}
Before stating the main results, let us give a definition of strong
solutions to (\ref{nonconvex target equation})-(\ref{ic nc}).

\begin{definition}
\label{solution to nonconvex target}A function $u\in C\left(  [0,T];V\right)
$ is  said to be  a \emph{strong solution}\ of \emph{(\ref{nonconvex target equation}%
)-(\ref{ic nc}) }if the following  conditions  are satisfied\/{\rm :}

\begin{description}
\item[(i)] $u\in L^{m}(0,T;X)\cap W^{1,p}(0,T;V)$, $\mathrm{d}_{V}%
\psi(u^{\prime})\in L^{p^{\prime}}(0,T;V^{\ast})$,

\item[(ii)] there exist $\eta^{1},\eta^{2}\in L^{p^{\prime}}(0,T;V^{\ast})$
such that $\eta^{1}\in\partial_{V}\varphi^{1}(u)$ and\ $\eta^{2}\in
\partial_{V}\varphi^{2}(u)$  a.e.~in $(0,T)$,  

\item[(iii)] $\mathrm{d}_{V}\psi(u^{\prime})+\eta^{1}-\eta^{2}-f(u)=0$ in
$V^{\ast}$ a.e.~in $(0,T)$, and $u(0)=u_{0}$.
\end{description}
\end{definition}

 Concerning the elliptic-in-time regularization of (\ref{nonconvex target equation})-(\ref{ic nc}), we shall treat the following weak formulation,
\begin{gather}
-\varepsilon\xi^{\prime}+\xi+\eta^{1}-\eta^{2}-f(u)=0 \  \mbox{ in } X^*  \ \text{ a.e. in
}(0,T)\text{,}\label{nonconvex euler eq}\\
\xi=\mathrm{d}_{V}\psi(u^{\prime})\text{, } \quad \eta^{1}\in\partial_{X}\varphi
_{X}^{1}(u)\text{, } \quad \eta^{2}\in\partial_{V}\varphi^{2}(u),\\
u(0)=u_{0}\text{, } \quad \xi(T)=0\text{.} \label{nc el ic}%
\end{gather}
Here, we are concerned with strong solutions of (\ref{nonconvex euler eq})-(\ref{nc el ic}) defined as follows: 
\begin{definition}
\label{def sol eu}A function $u\in C([0,T];V)$ is said to be a \emph{strong
solution} for \emph{(\ref{nonconvex euler eq})-(\ref{nc el ic})} if it
satisfies the following conditions\/{\rm :}

\begin{description}
\item[(i)] $u\in L^{m}(0,T;X)\cap W^{1,p}\left(  0,T;V\right)$,
$\xi=\mathrm{d}_{V}\psi(u)\in L^{p^{\prime}}\left(  0,T;V^{\ast}\right)$,
and $\xi^{\prime}\in L^{m^{\prime}}(0,T;X^{\ast})+L^{p^{\prime}}(0,T;V^{\ast
})$,

\item[(ii)] there exist $\eta^{1}\in L^{m^{\prime}}(0,T;X^{\ast})$, $\eta
^{2}\in L^{p^{\prime}}(0,T;V^{\ast})$ such that $\eta^{1}\in\partial
	   _{X}\varphi_{X}^{1}(u)$,\ $\eta^{2}\in\partial_{V}\varphi^{2}(u)$ and  the following holds true\/{\rm :}
	   \begin{gather*}
	    -\varepsilon\xi^{\prime}+\xi+\eta^{1}-\eta^{2}-f(u)=0 \ \text{ in } X^{\ast} \ \text{ a.e.~in } (0,T),\\
	    u(0)=u_{0}, \quad \xi(T)=0.
	   \end{gather*}  
\end{description}
\end{definition}
 We start with the case of convex energy functionals,  namely $\varphi^{2}=0$. Our first result reads,

\begin{theorem}
\label{thm convex}Let assumptions \emph{\textbf{(A1)}-\textbf{(A5)}} and
\emph{(\ref{ic}) }be satisfied with $\varphi^{2}=0$. Then, there exists
$\varepsilon_{0}>0$ such that for every $\varepsilon\in\left(  0,\varepsilon
_{0}\right)  $ the  elliptic-in-time regularization  \emph{(\ref{nonconvex euler eq})-(\ref{nc el ic})} admits  strong  solutions $u_{\varepsilon}$ in the sense of Definition
\emph{\ref{def sol eu}}. Moreover, there exists a sequence $\varepsilon
_{n}\rightarrow0$ such that $u_{\varepsilon_{n}}\rightarrow u$ weakly in
$L^{m}(0,T;X)\cap W^{1,p}\left(  0,T;V\right)  $ and strongly in $C\left(
[0,T];V\right)  $ and the limit $u$ solves  the target equation  
\emph{(\ref{nonconvex target equation})-(\ref{ic nc})} in the sense of
Definition \emph{\ref{solution to nonconvex target}.}
\end{theorem}

A proof will be given in Section \ref{convex case}.  Moreover, the assertion of Theorem \ref{thm convex} will be extended to nonconvex energy functionals  $\phi=\varphi^{1}-\varphi^{2}$ for $\varphi^{2}$ satisfying \textbf{(A6)}.
More precisely, we have the following:

\begin{theorem}
\label{full gen theorem}Let assumptions \emph{\textbf{(A1)}-\textbf{(A6)}} and
\emph{(\ref{ic}) }be satisfied. Then, the assertion of Theorem
\emph{\ref{thm convex}} holds true.
\end{theorem}

A proof of Theorem \ref{full gen theorem}, which will be presented in
Section \ref{nonconvex energies}, is based on an approximation  of subdifferential operators to reduce the problem into a convex energy case  and an application of Theorem \ref{thm convex} (to be more precise, Proposition \ref{main teo} below).  

\section{ Convex energies: Proof of Theorem \ref{thm convex}  \label{convex case}}

In this section, we  treat only convex energies, i.e., $\varphi^{2}=0$,  and prove Theorem \ref{thm convex}. 
We start  by  showing
existence of  strong  solutions to  the elliptic-in-time regularized  equation,%
\begin{gather}
-\varepsilon\xi^{\prime}+\xi+\eta  =f(u),\label{Eu}\\
\xi =\mathrm{d}_{V}\psi(u^{\prime})\text{,} \quad \eta \in\partial_{X}\phi_{X}(u)\text{,}\\
u(0)   =u_{0}, \quad  \xi(T) = 0  \label{Eu3}%
\end{gather}
for $\varepsilon>0$ small enough. The strategy of our proof  relies on  a variational technique based on the minimization of WED functionals (see~\cite{Ak-St3} and~\cite{Ak-St2})  as well as  a fixed-point argument.

\subsection{A fixed-point argument\label{section fixed point}}

Let us define the map $S:L^{p}(0,T;V)\rightarrow L^{p}(0,T;V)$ by
\[
S:v\mapsto w:= f(v) \longmapsto u\text{,}%
\]
where $u$ is the  unique  global minimizer of the WED functional $I_{\varepsilon,w}:L^{p} (0,T;V)\rightarrow(-\infty,\infty]$ defined by%
\begin{equation}
I_{\varepsilon,w}(u)=\left\{
\begin{array}
[c]{cl}%
\int_{0}^{T}\exp(-t/\varepsilon)(\varepsilon\psi(u^{\prime})+\phi
(u)-\left\langle w,u\right\rangle _{V}) & \text{if }\ u\in K(u_{0})\cap
L^{m}(0,T;X)\text{,}\\
\infty & \text{ otherwise, }%
\end{array}
\right.  \label{wed funct}%
\end{equation}
over the set $K(u_{0}):=\{u\in W^{1,p}\left(  0,T;V\right)  :u(0)=u_{0}\}$.
For the well-posedness of the map $S$ (namely,  existence and uniqueness  of a minimizer $u=\arg\min I_{\varepsilon,w}$), we  employ the following fact (see~\cite[Theorem 5.1]{Ak-St3}): 

 \begin{theorem}
\label{min th}Let $w\in L^{p^{\prime}}(0,T;V^{\ast})$ and
\textbf{\emph{(A1)-(A4)}} and \eqref{ic} be satisfied with $\varphi^2 \equiv 0$. Then, for all $\varepsilon>0$, the WED functional $I_{\varepsilon,w}$ defined by \emph{(\ref{wed funct})} admits at least one minimizer $u_{\varepsilon}$  such that 
\begin{align}
u_{\varepsilon}  &  \in L^{m}(0,T;X)\cap W^{1,p}\left(  0,T;V\right)
\text{,}\nonumber\\
\xi_{\varepsilon}  &  =\mathrm{d}_{V}\psi(u_{\varepsilon})\in L^{p^{\prime}%
}\left(  0,T;V^{\ast}\right)  \text{ and }\ \xi_{\varepsilon}^{\prime}\in
L^{m^{\prime}}(0,T;X^{\ast})+L^{p^{\prime}}(0,T;V^{\ast}).\nonumber
\end{align}
Furthermore, there exists $\eta_{\varepsilon}\in L^{m^{\prime}}(0,T;X^{\ast})$
such that $\eta_{\varepsilon}\in\partial_{X}\phi_{X}(u_{\varepsilon})$ and
$(u_{\varepsilon},\xi_{\varepsilon},\eta_{\varepsilon})$ satisfies%
\begin{gather}
-\varepsilon\xi_{\varepsilon}^{\prime}+\xi_{\varepsilon}+\eta_{\varepsilon}
 =w\text{ in }X^{\ast}\text{ a.e.~in }(0,T)\text{,}%
\label{fix point Euler}\\
u_{\varepsilon}(0) =u_{0}, \quad \xi_{\varepsilon
}(T)=0\text{.} \label{fix point ic}%
\end{gather}
In addition, if $\phi$ is strictly convex, then the minimizer of $I_{\vep,w}$ is unique. 
 \end{theorem}

 \begin{remark}{\rm
Note that  it is not restrictive to assume the strict convexity of $\phi=\varphi^{1}$.  Indeed, given $\varphi^{1}$ and $\varphi^{2}$ satisfying assumptions
\textbf{(A3)-(A4)}, \textbf{(A6)}, we define
$\tilde{\varphi}^{1}$ and $\tilde{\varphi}^{2}$ by
\begin{align*}
\tilde{\varphi}^{1}(u)  &  =\varphi^{1}(u)+|u|_{V}^{m-\delta}\text{,}\\
\tilde{\varphi}^{2}(u)  &  =\varphi^{2}(u)+|u|_{V}^{m-\delta}\text{}%
\end{align*}
for all $u\in V$ and some $\delta\in(0,m-1)$. Note that $\tilde{\varphi}^{1}$
and $\tilde{\varphi}^{2}$ satisfy assumptions \textbf{(A3)-(A4)},
\textbf{(A6)} and $\phi=\varphi^{1}-\varphi^{2}=\tilde{\varphi}%
^{1}-\tilde{\varphi}^{2}$. Moreover, $\tilde{\varphi}^{1}$ is strictly convex.
 }
\end{remark}

 The goal of this subsection  is now to prove that $S$ has a fixed point. More precisely, we  shall  prove
the following:

\begin{proposition}
\label{main teo}Let assumptions \emph{\textbf{(A1)}-\textbf{(A4)}} and \eqref{ic} be satisfied with $\varphi^2 \equiv 0$. Let $f$ satisfy condition \emph{(\ref{growth f})} and assume that $f$ is demicontinuous from $L^{p}(0,T;V)$ to $L^{p^{\prime}}(0,T;V^{\ast})$, i.e.,
\begin{equation}
 \left.
  \begin{array}{l}
  \mbox{if }
   u_{n}\rightarrow u\text{ strongly in }L^{p}(0,T;V),\\
   \mbox{then } f(u_{n}) \to f(u)\text{ weakly in }L^{p^{\prime}}(0,T;V^{\ast})\text{.}
\end{array}
 \right\}
\label{demicontinuity f}%
\end{equation}
Then, there exists $\varepsilon_{0}>0$ such
that for all $\varepsilon\in(0,\varepsilon_{0})$ the map $S$ has at least one
fixed point  $u_\varepsilon$. Moreover, such a fixed point is a strong solution  to the
 elliptic-in-time regularized  equation \emph{(\ref{Eu})-(\ref{Eu3})}.
\end{proposition}

Here we remark that \eqref{demicontinuity f} is weaker than the continuity of $f : V \to V^*$. Moreover, the demiclosedness will be essentially required for the nonconvex energy case in Section \ref{nonconvex energies}.

To this end, we shall simply check several assumptions to apply the Schaefer
fixed-point theorem (see Theorem \ref{schaefer} in Appendix) to the map $S$. Our proof is
divided into several steps.

\subsubsection{A priori estimates.}\label{Sss:3.1.1}

We shall now derive some (uniform in $\varepsilon$) estimates for the
solution of (\ref{fix point Euler})-(\ref{fix point ic}). Throughout this
section, the symbols $C$ and $c$ will denote some positive constants
independent of $\varepsilon$ which may vary even within the same line.

Fix $\varepsilon>0$,  $v\in L^p(0,T;V)$, and $w=f(v)$  and let $u:=u_{\varepsilon}$  be  the solution to
(\ref{fix point Euler})-(\ref{fix point ic}) given by Theorem \ref{min th}. Since $u^{\prime}\in
D(\partial_{V}\psi)$ and $\xi\in\partial_{V}\psi(u^{\prime})$ (indeed 
$D(\partial_V\psi)=V$ and 
$\partial_V\psi(v)=\{{\rm{d}}_V\psi(v) \}$ for all $v\in V$), by defining the
\emph{Fenchel conjugate} $\psi^{\ast}$ of $\psi$ by
\[
\psi^{\ast}(v)=\sup_{w\in V}\{\left\langle v,w\right\rangle _{V}%
-\psi(w)\} \quad \text{ for } \ v\in V^{\ast}%
\]
and by using the \emph{Fenchel identity,}%
\[
\psi(w)+\psi^{\ast}(v)=\left\langle v,w\right\rangle _{V} \ \Leftrightarrow \ %
w\in\partial_{V^{\ast}}\psi^{\ast}(v) \ \Leftrightarrow \ v\in\partial_{V}%
\psi(w)\text{,}%
\]
we have $u^{\prime}\in\partial_{V^{\ast}}\psi^{\ast}(\xi)$. Thus,
$\left\langle \xi^{\prime},u^{\prime}\right\rangle_{V}=\frac{\mathrm{d}%
}{\mathrm{d}t}\psi^{\ast}(\xi)$ by a chain-rule for subdifferentials. Testing equation (\ref{fix point Euler}) with
$u^{\prime}$ and integrating  both sides  over $(0,t)$, one gets%
\begin{equation}
-\varepsilon\int_{0}^{t}\frac{\mathrm{d}}{\mathrm{d}t}\psi^{\ast}(\xi
)+\int_{0}^{t}\left\langle \xi,u^{\prime}\right\rangle _{V}+\int_{0}^{t}%
\frac{\mathrm{d}}{\mathrm{d}t}\phi(u)  =  \int_{0}^{t}\left\langle
f(v),u^{\prime}\right\rangle _{V}\text{.} \label{in}%
\end{equation}

\begin{remark}{\rm
\label{remark formal}The above argument is formal. A rigorous
derivation of (\ref{in}) can be found in~\cite{Ak-St1},~\cite{Ak-St2}, and~\cite{Ak-St3}. 
 }
\end{remark}

As a consequence of  assumptions  (\ref{assumptions consequence}) and (\ref{growth f}), it
follows that%
  \begin{align}
-\varepsilon\psi^{\ast}(\xi(t))+\varepsilon\psi^{\ast}(\xi(0))+c\int_{0}%
^{t}|u^{\prime}|_{V}^{p}+\phi\left(  u\left(  t\right)  \right)  -\phi\left(
  u_{0}\right) \label{u prime test fix v}\\
  \leq C+\frac{c}{2}\int_{0}^{t}|u^{\prime}|_{V}^{p}+C\int_{0}%
^{t}|v|_{V}^{p}\text{.}\nonumber 
  \end{align}
As $\phi$ and  $\psi^\ast$  are bounded from below, by $u_{0}\in D(\phi)$, we have%
\begin{equation}
\frac{c}{2}\int_{0}^{t}|u^{\prime}|_{V}^{p}+\phi\left(  u\left(  t\right)
\right)  \leq C+C\int_{0}^{t}|v|_{V}^{p}+\varepsilon\psi^{\ast}%
(\xi(t))\label{est1}
\end{equation}
and,  recalling $\xi\left(  T\right)  =0$,
\begin{equation}
\frac{c}{2}\int_{0}^{T}|u^{\prime}|_{V}^{p}+\phi\left(  u\left(  T\right)
\right)  \leq C+C\int_{0}^{T}|v|_{V}^{p}\text{.} \label{est2}%
\end{equation}
Note that 
\begin{equation*}
\int_{0}^{t}\frac{\mathrm{d}}{\mathrm{d}t}|u|_{V}^{p}=\int_{0}%
^{t}p|u|_{V}^{p-2} \langle F_V u, u' \rangle_V  \leq\frac{c}{2}\int_{0}^{t}|u^{\prime}|_{V}%
^{p}+C\int_{0}^{t}|u|_{V}^{p}, 
\end{equation*}
where $F_{V}:V\rightarrow V^{\ast}$ denotes the duality mapping between $V$ and $V^{\star}$. Hence, substituting it into (\ref{est1}),  we obtain 
\begin{equation}
|u(t)|_{V}^{p}+\phi\left(  u\left(  t\right)  \right)  \leq C+C\int_{0}%
^{t}|v|_{V}^{p}+C\int_{0}^{t}|u|_{V}^{p}+\varepsilon\psi^{\ast}(\xi
(t))\text{.} \label{estimate 1}%
\end{equation}
Applying Gronwall's lemma (cf. Lemma \ref{gronwall} in Appendix), one
gets
\[
|u(t)|_{V}^{p}\leq C+ C\int_0 ^t \left(|v|_{V}^{p}+ \int_{0}^{t}|v|_{V}^{p} \right) +\varepsilon\psi^{\ast}%
(\xi(t))+C\varepsilon \int_{0}^{t}\psi^{\ast}(\xi)\text{. }%
\]
By substituting it into (\ref{estimate 1}), integrating both sides\ over
$[0,T]$ and taking the sum with (\ref{est2}), we get
\begin{align}
\frac{c}{2}\int_{0}^{T}|u^{\prime}|_{V}^{p}+\int_{0}^{T}|u|_{V}^{p}%
+\phi\left(  u\left(  T\right)  \right)  +\int_{0}^{T}\phi\left(  u\right)
 \label{estimate 3}\\%
 \leq C+C\int_{0}^{T}|v|_{V}^{p}+C\varepsilon\int_{0}^{T}\psi^{\ast}(\xi).\nonumber
\end{align}
We now show that
\begin{equation}
\psi^{\ast}(\xi) \leq C|u^{\prime}|_{V}^{p}+C\text{.}
\label{fenchel estimate}%
\end{equation}
Indeed, by definition $\psi^{\ast}(\xi)=\sup_{w\in V}\left(  \left\langle
\xi,w\right\rangle_{ V } -\psi(w)\right)  $ and by using assumption  \textbf{(A1)},  for any $\delta>0$, 
one can take a constant $C_{\delta}>0$ such that 
\begin{align*}
\psi^{\ast}(\xi)  &  \leq\sup_{w\in V} \left\{|\xi|_{V^{\ast}}|w|_{V}+1-\frac
{1}{C_{1}}|w|_{V}^{p} \right\}\\
&  \leq\sup_{w\in V} \left\{C_{\delta}|\xi|_{V^{\ast}}^{p^{\prime}}+\delta
|w|_{V}^{p}+1-\frac{1}{C_{1}}|w|_{V}^{p} \right\}.
\end{align*}
Choosing $\delta=\frac{1}{C_{1}%
}$ and using  assumption  \textbf{(A2)}, we get
\begin{align*}
 \psi^{\ast}(\xi) \leq\sup_{w\in V}\{C|\xi|_{V^{\ast}}^{p^{\prime}}+1\}
  = C|\xi|_{V^{\ast}}^{p^{\prime}}+1\leq C|u^{\prime}|_{V}^{p}+C\text{. }%
\end{align*}
Thus, substituting it into (\ref{estimate 3}), we can choose $\varepsilon
=\varepsilon(T,\psi)$ (depending on $\psi$ and $T$, but not on $\phi$)
sufficiently small to obtain
\begin{equation}
\int_{0}^{T}|u|_{V}^{p}+\int_{0}^{T}|u^{\prime}|_{V}^{p}+\int_{0}^{T}%
\phi\left(  u\right)  \leq C+C\int_{0}^{T}|v|_{V}^{p}\text{.}
\label{basic estimate}%
\end{equation}
Therefore, $u$ is uniformly bounded in $W^{1,p}(0,T;V)$ and hence in $C([0,T];V)$ by 
$C+C\int_{0}^{T}|v|_{V}^{p}$.
As a consequence of assumption \textbf{(A3)}, we  obtain the following  estimate:
\begin{equation}
\left\Vert u\right\Vert _{W^{1,p}(0,T;V)}^{p}+\left\Vert u\right\Vert
_{L^{m}(0,T;X)}^{m}\leq C(1+ \ell_3 ( C+C \Vert v \Vert_{L^p(0,T;V)} ))(1+\int_{0}^{T}|v|_{V}^{p}). \label{a priori estimate}%
\end{equation}

 \begin{remark}
  {\rm
We can prove estimate (\ref{a priori estimate}) in
an easier way. Indeed, using the nonnegativity of $\phi$ and $u(0)=u_{0}\in V$,
we can deduce  from (\ref{est2})  that $\left\Vert u' \right\Vert_{L^p(0,T;V)}^{p} \leq C+C\int_{0}^{T}|v|_{V}^{p}$. By substituting (\ref{fenchel estimate}) into (\ref{est1}) and by integrating both sides over $\left(  0,T\right)$, we get $\int_0^T \phi(u)\leq C+C\int_{0}^{T}|v|_{V}^{p}$ and thus (\ref{a priori estimate}) by virtue of \textbf{(A3)}. However, the argument  starting from (\ref{estimate 1}) to derive  estimate (\ref{a priori estimate}) will be used later (see  (\ref{new estimate}) and (\ref{estimate *}) below).
  }
 \end{remark}

\subsubsection{The map  $S:L^{p}(0,T;V)\rightarrow L^{p}(0,T;V)$  is continuous.}\label{sec S cont}

We recall that $S$ is the composition of two maps: $S:v\mapsto w:= f(v)\longmapsto
u$. We notice that, as a consequence of (\ref{demicontinuity f}), the map $v\longmapsto f(v)$ is demicontinuous from $L^{p}(0,T;V)$ into $L^{p^{\prime}}(0,T;V^{\ast})$. Thus, we are only left to prove that the solution operator $w\mapsto u$ is weakly-strongly continuous from $L^{p^{\prime}}(0,T;V^{\ast})$ into $L^{p}(0,T;V)$. Let $\{w_{h}\}$ be a sequence  in $L^{p^{\prime}}(0,T;V^{\ast})$ such that $w_{h}\rightarrow w$ weakly in $L^{p^{\prime}}(0,T;V^{\ast})$  as  $h\rightarrow0$. Then, there exists $C$ independent of $h$ such that  $\left\Vert w_{h}\right\Vert _{L^{p^{\prime}}(0,T;V^{\ast})}\leq C$  and estimate (\ref{a priori estimate}) implies that the family $\{u_{h}=\operatorname{argmin} I_{\varepsilon,w_{h}}\}$  of minimizers  is uniformly bounded in $W^{1,p}(0,T;V)\cap L^{m}(0,T;X)\hookrightarrow \hookrightarrow L^{p}(0,T;V)$, and hence,  we deduce that,  up to a subsequence $k_{h}$, $u_{k_{h}}\rightarrow u$ strongly in $L^{p}(0,T;V)$. In order to identify the limit $u$, we prove that the corresponding WED functionals $I_{\varepsilon ,w_{h}}$ converge to $I_{\varepsilon,w}$ in the sense of $\Gamma$-convergence (see, e.g, \cite{DMa}). Indeed let $\{u_{h}\}\in L^{p}(0,T;V)$ be such that $u_{h}\rightarrow u$ strongly in $L^{p}(0,T;V)$. Then, as $\psi$ and $\phi$ are convex and l.s.c.,  we find that 
\begin{align*}
\liminf_{h\rightarrow0}I_{\varepsilon,w_{h}}(u_{h})  &  =\liminf
_{h\rightarrow0}\int_{0}^{T}\exp(-t/\varepsilon)\left(  \varepsilon\psi
(u_{h}^{\prime})+\phi(u_{h})-\left\langle w_{h},u_{h}\right\rangle _{V}\right)
\\
&  \geq 
 I_{\varepsilon,w}(u)\text{.} %
\end{align*}
As for the  existence of a  recovering sequence for each  $u\in K(u_{0})\cap L^{m}(0,T;X)$, we simply set $u_{h}  \equiv  u$. 
Then, one can immediately check that%
\[
\lim_{h\rightarrow0}I_{\varepsilon,w_{h}}(u)=I_{\varepsilon,w}(u).
\]
As a consequence of the $\Gamma$-convergence of the WED functionals  along with  the convergence $u_{k_{h}}\rightarrow u$ strongly in $L^{p}(0,T;V)$, we deduce
that $u$ minimizes $I_{\varepsilon,w}$. We recall that for every $w$ the
minimizer of $I_{\varepsilon,w}$ is unique (due to the strict convexity of
$I_{\varepsilon,w}$). Thus, the convergence holds for the whole sequence
$\{u_{h}\}$. This proves continuity of $S$.

\subsubsection{Compactness.}

We here prove  the compactness of  the map $S:L^{p}(0,T;V)\rightarrow L^{p}(0,T;V)$. Let $\left\{  v_{h}\right\}  \subset L^{p}(0,T;V)$ be a bounded
sequence and let  $u_{h}$ be the minimizer of $I_{\varepsilon,f(v_{h})}$ . Then, as a consequence of estimate\ (\ref{a priori estimate}), the family  $\{u_{h}\}$  is uniformly bounded in $W^{1,p}(0,T;V)\cap L^{m}(0,T;X)\hookrightarrow
\hookrightarrow L^{p}(0,T;V)$, and hence, up to a subsequence, $u_{h}%
\rightarrow u$ strongly in $L^{p}(0,T;V)$.

\subsubsection{Boundedness of $\{  \bar{v}  \in L^{p}(0,T;V):  \bar{v}  =\alpha S(  \bar{v}  )$ for $\alpha \in\lbrack0,1]\}$.}\label{Sss:Bounded}

 In order  to apply the Schaefer fixed-point theorem (see Theorem \ref{schaefer} in Appendix), 
we are only left to prove that the set $A:=\{ \bar{v}  \in L^{p}%
(0,T;V): \bar{v}  =\alpha S( \bar{v}  )$ for $\alpha\in\lbrack0,1]\}$ is bounded. Note that $A$ is bounded if and only if $\{  \bar{v}  \in L^{p}(0,T;V):  \bar{v}  /\alpha=S( \bar{v} )$ for $\alpha
\in(0,1]\}$ is bounded. Hence, we shall prove that $B=\{v\in L^{p}%
(0,T;V):v=S(\alpha v)$ for $\alpha\in(0,1]\}$ is bounded. This yields the
boundedness of $A$.

In case $B=\emptyset$, we immediately find that $A=\{0\}$, and hence, nothing remains  to be proved. In case $B\neq\emptyset$, let $v\in B$. Then, there exists $\alpha
\in(0,1]$ such that $v=S(\alpha v)$. Let $u$ be the minimizer of
$I_{\varepsilon,f(\alpha v)}$, i.e. $u=S(\alpha v)$ $(=v)$. Then, $u$ solves
\begin{equation}
-\varepsilon\xi^{\prime}+\xi+\eta=f(\alpha v)=f(\alpha u)\text{ in }X^{\ast
}\text{ \ a.e. in }\left(  0,T\right)  \text{,} \label{euler alpha}%
\end{equation}
where $\xi=\mathrm{d}_{V}\psi(u)$ and $\eta\in\partial_{X}\phi_{X}(u)$ a.e. in
$\left(  0,T\right)  $. We shall prove that solutions to (\ref{euler alpha})
are bounded in $L^{p}(0,T;V)$ uniformly in $\alpha$ for $\varepsilon$ small
enough. Testing (\ref{euler alpha}) with $u^{\prime}$ and integrating  both sides  over
$(0,t)$, we get
\[
-\varepsilon\int_{0}^{t}\frac{\mathrm{d}}{\mathrm{d}t}\psi^{\ast}(\xi
)+\int_{0}^{t}\left\langle \xi,u^{\prime}\right\rangle _{V}+\int_{0}^{t}%
\frac{\mathrm{d}}{\mathrm{d}t}\phi(u)=\int_{0}^{t}\left\langle f(\alpha
u),u^{\prime}\right\rangle _{V}\text{}%
\]
 (see also Remark \ref{remark formal}).
As a consequence of (\ref{assumptions consequence}) and \textbf{(A5)}, we get%
\begin{align}
\lefteqn{ -\varepsilon\psi^{\ast}(\xi(t))+\varepsilon\psi^{\ast}(\xi(0))+c\int
_{0}^{t}|u^{\prime}|_{V}^{p}+\phi\left(  u\left(  t\right)  \right)
-\phi\left(  u_{0}\right) \label{new estimate}}\\
 & \leq C+\frac{c}{2}\int_{0}^{t}|u^{\prime}|_{V}^{p}+C\int_{0}^{t}|f(\alpha
 u)|_{V^{\ast}}^{p^{\prime}}
 \nonumber\\
 &\leq C+\frac{c}{2}\int_{0}^{t}|u^{\prime}|_{V}^{p}+C\int_{0}^{t}|u|_{V}%
^{p}\text{.}  \nonumber
\end{align}
  Proceeding as in \S \ref{Sss:3.1.1}, one can particularly derive formula (\ref{estimate 1}) with $v=u$ and thus (for $\varepsilon>0$ small enough) formula (\ref{basic estimate})-(\ref{a priori estimate}) with $v=0$. In particular, $\left\Vert u\right\Vert _{L^{p}(0,T;V)}\leq C$, where $C$ does
not depend on $\varepsilon$ and $\alpha$.

Thanks to the Schaefer fixed-point theorem (see Theorem \ref{schaefer} in
Appendix), the map $S$ has a fixed point $u_\vep$ for $\varepsilon>0$ small enough.  Furthermore, the fixed point $u_\vep$ of $S$ solves  (\ref{Eu})-(\ref{Eu3}) and  satisfies  the relation,
\[
u_{\varepsilon}=\underset{  \tilde{u} \in K(u_{0})}{ \operatorname{argmin}}\int_{0}^{T}\exp(-t/\varepsilon
)(\varepsilon\psi( \tilde{u} ^{\prime})+\phi( \tilde{u} )-\left\langle f(u_{\varepsilon
}), \tilde{u} \right\rangle _{V})\text{. }%
\]
This completes the proof of Proposition \ref{main teo}.

\subsection{The causal limit\label{causal limit}}

We shall now prove that solutions  $u_\vep$ of the elliptic-in-time regularized equations  (\ref{Eu})-(\ref{Eu3}) converge, up to a subsequence, to  a solution of the target  equation
(\ref{nonconvex target equation})-(\ref{ic nc}), in case $\phi$ is convex
(i.e. $\varphi^{2}=0$). More precisely, we prove the following proposition:

\begin{proposition}
Assume  \textbf{\emph{(A1)-(A5)}}, \eqref{ic} and $\varphi^{2}=0$. For  each  $\varepsilon>0$ small enough, let $u_{\varepsilon}$ be a solution  of the elliptic-in-time regularized equation  \emph{(\ref{Eu})-(\ref{Eu3}).} Then, there exist a sequence $\varepsilon_{n}\rightarrow0$ and a limit $u\in L^{m}(0,T;X)\cap W^{1,p}(0,T;V)$ such that $u_{\varepsilon_{n}}\rightarrow u$ weakly in $L^{m}(0,T;X)\cap W^{1,p}(0,T;V)$ and strongly in $C([0,T];V)$ and  the limit  $u$ solves \emph{(\ref{nonconvex target equation})-(\ref{ic nc})}.
\end{proposition}

Our proof is divided into two steps.

\subsubsection{A priori uniform estimates.}

Testing equation (\ref{Eu}) with $u' =u^{\prime}_\vep$ and integrating it over $(0,t)$, we get
\[
-\varepsilon\int_{0}^{t}\frac{\mathrm{d}}{\mathrm{d}t}\psi^{\ast}%
(\xi_{\varepsilon})+\int_{0}^{t}\left\langle \xi_{\varepsilon},u_{\varepsilon
}^{\prime}\right\rangle _{V}+\int_{0}^{t}\frac{\mathrm{d}}{\mathrm{d}t}%
\phi(u_{\varepsilon})=\int_{0}^{t}\left\langle f(u_{\varepsilon}%
),u_{\varepsilon}^{\prime}\right\rangle _{V}\text{}%
\]
(see also Remark \ref{remark formal}). As a consequence of assumption (\ref{assumptions consequence}), we obtain%
\begin{align*}
& -\varepsilon\psi^{\ast}(\xi_{\varepsilon}(t))+\varepsilon\psi^{\ast}%
(\xi_{\varepsilon}(0))+\int_{0}^{t}|u_{\varepsilon}^{\prime}|_{V}^{p}%
+\phi\left(  u_{\varepsilon}\left(  t\right)  \right)  -\phi\left(
u_{0}\right)  \\
& \leq C+\frac{1}{2}\int_{0}^{t}|u_{\varepsilon}^{\prime}|_{V}%
^{p}+C\int_{0}^{t}|f(u_{\varepsilon})|_{V^{\ast}}^{p^{\prime}}.
\end{align*}
By repeating the same argument as in \S \ref{Sss:Bounded} (with $\alpha = 1$), we can obtain, for $\varepsilon>0$ sufficiently small,
the estimate,
\[
\left\Vert u_{\varepsilon}\right\Vert _{W^{1,p}(0,T;V)}+\left\Vert
u_{\varepsilon}\right\Vert _{L^{m}(0,T;X)}\leq C.
\]
Due to assumptions \textbf{(A1)-(A5),} we have
\begin{align*}
\left\Vert \eta_{\varepsilon}\right\Vert _{L^{m^{\prime}}(0,T;X^{\ast})}  &
\leq C,\\
\left\Vert \xi_{\varepsilon}\right\Vert _{L^{p^{\prime}}(0,T;V^{\ast})}  &
\leq C\text{,}\\
\left\Vert f(u_{\varepsilon})\right\Vert _{L^{p^{\prime}}(0,T;V^{\ast})}  &
\leq C\text{,}%
\end{align*}
and, by comparison  of each term  in (\ref{Eu}),
\[
\left\Vert \varepsilon\xi_{\varepsilon}^{\prime}\right\Vert _{L^{p^{\prime}%
}(0,T;V^{\ast})+L^{m^{\prime}}(0,T;X^{\ast})}\leq C\text{.}%
\]

\subsubsection{The passage to the limit.}

From the a priori estimate obtained above, we can derive, along some not-relabeled subsequence, the following convergences:%
\begin{align*}
\eta_{\varepsilon}  &  \rightarrow\eta\text{ weakly in }L^{m^{\prime}%
}(0,T;X^{\ast}),\\
\xi_{\varepsilon}  &  \rightarrow\xi\text{ weakly in }L^{p^{\prime}%
}(0,T;V^{\ast}),\\
\varepsilon\xi_{\varepsilon}^{\prime}  &  \rightarrow0\text{ weakly in
}L^{p^{\prime}}(0,T;V^{\ast})+L^{m^{\prime}}(0,T;X^{\ast}),\\
u_{\varepsilon}  &  \rightarrow u\text{ weakly in }W^{1,p}(0,T;V)\text{ and in
}L^{m}(0,T;X)\text{,}\\
f(u_{\varepsilon})  &  \rightarrow f^{\ast}\text{ weakly in }L^{p^{\prime}%
}(0,T;V^{\ast})\text{,}%
\end{align*}
and hence,  thanks to Aubin-Lions-Simon's compactness lemma  (see \cite{Si}),
\begin{equation}
u_{\varepsilon}\rightarrow u\text{ strongly in }C\left(  [0,T];V\right)
\text{.} \label{strong convergence u}%
\end{equation}
In particular,%
\[
u_{\varepsilon}(t)\rightarrow u(t)\text{ strongly in }V\text{ for all }%
t\in\lbrack0,T]
\]
and $u\left(  0\right)  =u_{0}$. As a consequence of  the continuity of $f:V \to V^*$ 
(see also Remark \ref{remark assumptions f}) and the strong convergence (\ref{strong convergence u}), 
we have $f^{\ast}=f(u)$ and
\begin{equation}
f(u_{\varepsilon})\rightarrow f(u)\text{ strongly in }L^{p^{\prime}%
}(0,T;V^{\ast})\text{.} \label{strong convergence f}%
\end{equation}
Thus,%
\begin{equation}
\xi+\eta-f\left(  u\right)  =0\text{ in }X^{\ast}\text{ a.e.\ in }[0,T].
\label{target eq}%
\end{equation}
From the final condition $\xi_{\varepsilon}(T)=0$ and the convergences above,
it follows that  
\[
- \left\langle \varepsilon\xi_{\varepsilon}(t),v\right\rangle _{X}=\left\langle
\int_{t}^{T} \varepsilon\xi_{\varepsilon}^{\prime}(s)\mathrm{d}s,v\right\rangle
_{X}= \int_{t}^{T} \left\langle \varepsilon\xi_{\varepsilon}^{\prime
}(s),v\right\rangle _{X}\mathrm{d}s\rightarrow0
\]
for all $v\in X$, which yields
\[
\varepsilon\xi_{\varepsilon}(t)\rightarrow0\text{ weakly in }X^{\ast}\text{
for each }t\in\lbrack0,T]\text{.}%
\]
We next verify $\eta(t)\in\partial_{V}\phi(u(t))~$for a.e. $t\in\left(
0,T\right)  $.  Since $\eta$ and $u$ entail sufficient regularity,  thanks to \cite[Proposition 2.1]{Ak-St2}, it is sufficient to show  a (weak) relation  $\eta(t)\in\partial_{X}\phi_{X}(u(t))$ for a.e.~$t\in\left(  0,T\right)
$. By comparison in (\ref{Eu}), integrating by parts (cf.~\cite{Ak-St2} for a rigorous proof 
of the integration-by-parts formula) and using the convergences  obtained so far,  we have
\begin{align*}
\int_{0}^{T}\left\langle \eta_{\varepsilon},u_{\varepsilon}\right\rangle _{X}
&  =\int_{0}^{T}\left\langle \varepsilon\xi_{\varepsilon}^{\prime
},u_{\varepsilon}\right\rangle _{X}-\int_{0}^{T}\left\langle \xi_{\varepsilon
},u_{\varepsilon}\right\rangle _{V}+\int_{0}^{T}\left\langle f(u_{\varepsilon
}),u_{\varepsilon}\right\rangle _{V}\\
&  =-\varepsilon\left\langle \xi_{\varepsilon}(0),u_{0}\right\rangle _{X}%
-\int_{0}^{T}\left\langle \varepsilon\xi_{\varepsilon},u_{\varepsilon}%
^{\prime}\right\rangle _{V}-\int_{0}^{T}\left\langle \xi_{\varepsilon
},u_{\varepsilon}\right\rangle _{V}+\int_{0}^{T}\left\langle f(u_{\varepsilon
}),u_{\varepsilon}\right\rangle _{V}\\
&  \rightarrow-\int_{0}^{T}\left\langle \xi,u\right\rangle _{V}+\int_{0}%
^{T}\left\langle f(u),u\right\rangle _{V}\text{.}%
\end{align*}
Hence, we particularly get
\[
\limsup_{\varepsilon\rightarrow0}\int_{0}^{T}\left\langle \eta_{\varepsilon
},u_{\varepsilon}\right\rangle _{X}\leq-\int_{0}^{T}\left\langle
\xi,u\right\rangle _{V}+\int_{0}^{T}\left\langle f(u),u\right\rangle _{V}%
=\int_{0}^{T}\left\langle \eta,u\right\rangle _{X}\text{.}%
\]
By the demiclosedness of the maximal monotone operator $\partial_{X}\phi_{X}$
and by applying \cite[Proposition 1.1]{Ke}, we conclude that $\eta(t)\in
\partial_{X}\phi_{X}(u(t))$ for a.e. $t\in\left(  0,T\right)  $. Let us
finally show that $\xi(t)=\mathrm{d}_{V}\psi(u^{\prime}(t))$ for almost all
$t\in\left(  0,T\right)  $. Combining \cite[Theorem 5.1]{Ak-St3} and
\cite[Theorem 3.3, Lemma A.1]{Ak-St2}, we deduce the following inequality%
\begin{equation}
\int_{0}^{t}\left\langle \xi_{\varepsilon},u_{\varepsilon}^{\prime
}\right\rangle _{V}\leq\varepsilon\left\langle \xi_{\varepsilon}%
(t),u_{\varepsilon}^{\prime}(t)\right\rangle _{V}+\varepsilon\psi
(0)-\phi(u_{\varepsilon}(t))+\phi(u_{0})+\int_{0}^{t}\left\langle
f(u_{\varepsilon}),u_{\varepsilon}^{\prime}\right\rangle _{V}, \label{formula}%
\end{equation}
which can be formally obtained by substituting identity (\ref{Eu}) into the
left-hand side of (\ref{formula}) and by integrating it by parts. Thus, from the
convergences  above (in particular,  (\ref{strong convergence f})), using the lower semicontinuity of $\phi$ and recalling that $\eta\in L^{p}(0,T;V)$ (by comparison  of each term  in (\ref{target eq})) and that $\xi_\varepsilon(T)=0$, we deduce that
\begin{align*}
\limsup_{\varepsilon\rightarrow0}\int_{0}^{T}\left\langle \xi_{\varepsilon
},u_{\varepsilon}^{\prime}\right\rangle _{V}  &  \leq\lim_{\varepsilon
\rightarrow0}\varepsilon\psi(0)-\liminf_{\varepsilon\rightarrow0}\phi\left(
u_{\varepsilon}(T)\right)  +\phi(u_{0})+\lim_{\varepsilon\rightarrow0}\int
_{0}^{T}\left\langle f(u_{\varepsilon}),u_{\varepsilon}^{\prime}\right\rangle
_{V}\\
&  \leq-\phi\left(  u\left(  T\right)  \right)  +\phi\left(  u_{0}\right)
+\int_{0}^{T}\left\langle f(u),u^{\prime}\right\rangle _{V}\\
&  =-\int_{0}^{T}\left\langle \eta,u^{\prime}\right\rangle _{V}+\int_{0}%
^{T}\left\langle f(u),u^{\prime}\right\rangle _{V}\\
&  =\int_{0}^{T}\left\langle \xi,u^{\prime}\right\rangle _{V}\text{. }%
\end{align*}
Thus, $\xi(t)=\mathrm{d}_{V}\psi(u^{\prime}(t))$ for a.e. $t\in\left(
0,T\right)  $, and  hence,  $u$ solves (\ref{nonconvex target equation})-(\ref{ic nc}) with $\eta_2 =0$.

\section{Nonconvex energies\label{nonconvex energies}: Proof of Theorem
\ref{full gen theorem}}

 This section is devoted to  the proof of  Theorem \ref{full gen theorem} for general nonconvex energy functionals. To this end, we shall employ the \emph{Moreau-Yosida regularization} for convex functionals (equivalently, the \emph{Yosida approximation} for subdifferentials) to approximate the target equation in order to reduce the problem to the convex energy setting of Proposition \ref{main teo}. Finally, we shall pass to the limit of approximated solutions and obtain a solution to the target equation.

\subsection{The Moreau-Yosida regularization.}

We first regularize equation (\ref{nonconvex euler eq}): for every $\lambda
>0$, we define the Moreau-Yosida regularization of $\varphi^{2}$ by
\[
\varphi_{\lambda}^{2}\left(  u\right)  =\min_{v\in V}\left(  \frac{\lambda}%
{p}\left\vert \frac{u-v}{\lambda}\right\vert _{V}^{p}+\varphi^{2}(v)\right)
=\frac{\lambda}{p}\left\vert \frac{u-J_{\lambda}u}{\lambda}\right\vert
_{V}^{p}+\varphi^{2}(J_{\lambda}u)\text{,}%
\]
where $J_{\lambda}$ denotes the resolvent for $\partial_{V}\varphi^{2}$ (see \eqref{def:reso} in Appendix), and consider the  following  approximate equations for
(\ref{nonconvex euler eq})-(\ref{nc el ic}):
\begin{gather}
-\varepsilon\xi_{\varepsilon,\lambda}^{\prime}+\xi_{\varepsilon,\lambda}%
+\eta_{\varepsilon,\lambda}^{1}-\eta_{\varepsilon,\lambda}^{2}%
-f(u_{\varepsilon,\lambda})=0\text{ in }X^{\ast}\text{ a.e. in }%
(0,T)\text{,}\label{regularized nonconv euler eq}\\
\xi_{\varepsilon,\lambda}= \mathrm{d}_{V}\psi(u_{\varepsilon,\lambda}%
^{\prime})\text{, } \quad \eta_{\varepsilon,\lambda}^{1}\in\partial_{X}\varphi
_{X}^{1}(u_{\varepsilon,\lambda})\text{,\ } \quad \eta_{\varepsilon,\lambda}^{2}=\partial_{V}\varphi_{\lambda}^{2}(u_{\varepsilon,\lambda}),\\
u_{\varepsilon,\lambda}(0)=u_{0}\text{, }\quad \xi_{\varepsilon,\lambda
}(T)=0\text{.} \label{regularized nonconvex euler eq ic}%
\end{gather}
Note that $\partial_{V}\varphi_{\lambda}^{2}(u):V\rightarrow V^{\ast}$ is
single-valued and demicontinuous, and it satisfies  assumption  (\ref{growth f}) (see \S \ref{yosida regularization} in Appendix for details). Moreover, the mapping $Q_\lambda: u \mapsto \partial_{V}\varphi_{\lambda}^{2}(u)$ from $L^p(0,T;V)$ to $L^{p'}(0,T;V^*)$ also entails the demicontinuity (\ref{demicontinuity f}). Indeed, let $u_n \to u$ strongly in $L^p(0,T;V)$ and fix $ v  \in L^p(0,T;V)$. Then by demicontinuity of $\partial_{V}\varphi_{\lambda}^{2}(\cdot)$ from $V$ into $V^*$, we see that $\langle \partial_{V}\varphi_{\lambda}^{2}(u_n(t)),  v(t)  \rangle_V \to \langle \partial_{V}\varphi_{\lambda}^{2}(u(t)),  v(t)  \rangle_V$ for a.e.~$t \in (0,T)$. By virtue of \eqref{growth f} with $f = \partial_V \varphi^2_\lambda$, thanks to Vitali's convergence theorem, the function $t \mapsto \langle \partial_{V}\varphi_{\lambda}^{2}(u(t)),  v(t)  \rangle_V$ turns out to be integrable on $(0,T)$, and moreover, $\langle \partial_{V}\varphi_{\lambda}^{2}(u_n(\cdot)),  v (\cdot) \rangle_V \to \langle \partial_{V}\varphi_{\lambda}^{2}(u(\cdot)),  v (\cdot) \rangle_V$ strongly in $L^1(0,T)$. Therefore the mapping $Q_\lambda$ is demicontinuous from $L^p(0,T;V)$ into $L^{p'}(0,T;V^*)$. Applying Proposition \ref{main teo} with $f$ replaced by $f+\partial_{V}\varphi
_{\lambda}^{2}$, we deduce existence of (at least) a solution $u_{\varepsilon
,\lambda}$ to (\ref{regularized nonconv euler eq}%
)-(\ref{regularized nonconvex euler eq ic}).

\subsection{Uniform estimates.}

In order to pass to the limit  of solutions  as $\lambda\rightarrow0$, and then, as $\varepsilon\rightarrow0$, we  first establish  some (uniform in $\lambda$ and $\varepsilon$) estimates. Hereafter, $C$ and $c$ will denote positive constants not depending on $\varepsilon$ and $\lambda$ which may vary even within the same line. Testing equation (\ref{regularized nonconv euler eq}) with $u_{\varepsilon,\lambda}^{\prime}$ and integrating  it  over $(0,t)$, we get
\begin{align*}
-\varepsilon\int_{0}^{t}\frac{\mathrm{d}}{\mathrm{d}t}\psi^{\ast}
(\xi_{\varepsilon,\lambda})+\int_{0}^{t}\left\langle \xi_{\varepsilon,\lambda
},u_{\varepsilon,\lambda}^{\prime}\right\rangle_{V}+\int_{0}^{t}
\frac{\mathrm{d}}{\mathrm{d}t}\varphi^{1}(u_{\varepsilon,\lambda})-\int_{0}^{t}\frac{\mathrm{d}}{\mathrm{d}t}\varphi_{\lambda}^{2}(u_{\varepsilon,\lambda}) \\
=\int_{0}^{t}\left\langle f(u_{\varepsilon,\lambda}),u_{\varepsilon
,\lambda}^{\prime}\right\rangle _{V}
\end{align*}
(see Remark \ref{remark formal}). By virtue of (\ref{assumptions consequence}), we estimate
\begin{align*}
 -\varepsilon\psi^{\ast}(\xi_{\varepsilon,\lambda}(t))+\varepsilon\psi
^{\ast}(\xi_{\varepsilon,\lambda}(0))+c\int_{0}^{t}|u_{\varepsilon,\lambda
}^{\prime}|_{V}^{p} +\varphi^{1}(u_{\varepsilon,\lambda}\left(  t\right)
)-\varphi_{\lambda}^{2}(u_{\varepsilon,\lambda}(t))\\
 \leq \varphi^{1}(u_{0})-\varphi_{\lambda}^{2}(u_{0})+ C+\frac{c}{2}\int_{0}^{t}|u_{\varepsilon,\lambda}^{\prime}|_{V}%
^{p}+C\int_{0}^{t}|f(u_{\varepsilon,\lambda})|_{V^{\ast}}^{p^{\prime}}\text{,}%
\end{align*}
and thus, by using  assumption  \textbf{(A6)}, %
 \begin{align}
-\varepsilon\psi^{\ast}(\xi_{\varepsilon,\lambda}(t))+\varepsilon\psi
^{\ast}(\xi_{\varepsilon,\lambda}(0))+c\int_{0}^{t}|u_{\varepsilon,\lambda
  }^{\prime}|_{V}^{p}+(1-k)\varphi^{1}(u_{\varepsilon,\lambda}(t))
 \label{estimate *}\\
  \leq \varphi^{1}(u_{0})-\varphi_{\lambda}^{2}(u_{0}) +C \left( |u_{\vep,\lambda}|_V^p + 1 \right) +\frac{c}{2}\int_{0}^{t}|u_{\varepsilon,\lambda}^{\prime}|_{V}%
^{p}+C\int_{0}^{t}|f(u_{\varepsilon,\lambda})|_{V^{\ast}}^{p^{\prime}}.\nonumber \end{align}
Here we note that, for any $\mu > 0$ and $w \in W^{1,p}(0,T;V)$,
\begin{align*}
 \mu \dfrac{\d}{\d t} |w|_V^p
 &= p \mu  |w|_V^{p-1} \dfrac{\d}{\d t} |w|_V\\
 &\leq p \mu |w|_V^{p-1} |w'|_V
 \leq (p-1) \mu^{p'} |w|_V^p + |w'|_{V}^p,
\end{align*}
which yields,  by integrating over $(0,t)$,
$$
\mu |w(t)|_V^p \leq \mu |w(0)|_V^p + (p-1) \mu^{p'} \int^t_0 |w|_V^p + \int^t_0 |w'|_V^p \ \mbox{ for } \ t \in [0,T].
$$
Substituting the inequality above  into  \eqref{estimate *} with $\mu > 0$ large enough, we obtain
 \begin{align}
  \lefteqn{
-\varepsilon\psi^{\ast}(\xi_{\varepsilon,\lambda}(t))+\varepsilon\psi
^{\ast}(\xi_{\varepsilon,\lambda}(0))+\dfrac c 4\int_{0}^{t}|u_{\varepsilon,\lambda
  }^{\prime}|_{V}^{p}+(1-k)\varphi^{1}(u_{\varepsilon,\lambda}(t))
  }\label{estimate * 2}\\
 &\leq \varphi^{1}(u_{0}) +C \left(|u_0|_V^p + 1 \right)+ C \int_{0}^{t}|u_{\varepsilon,\lambda}|_{V}^{p}+C\int_{0}^{t}|f(u_{\varepsilon,\lambda})|_{V^{\ast}}^{p^{\prime}}\nonumber\\
 &\stackrel{\eqref{growth f}}\leq \varphi^{1}(u_{0}) +C \left(|u_0|_V^p + 1 \right)+ C \left( \int_{0}^{t}|u_{\varepsilon,\lambda}|_{V}^{p}+1\right).\nonumber
 \end{align}
Repeating the same argument as in the last section, we obtain (for $\varepsilon$ small enough)
 \[
 \int^T_0 |u_{\vep,\lambda}|_V^p + \int_{0}^{T}|u_{\varepsilon,\lambda}^{\prime}|_{V}^{p}+\int_{0}^{T}\varphi
^{1}\left(  u_{\varepsilon,\lambda}\right)  \leq C,
\]
and thus%
\begin{align}
\left\Vert u_{\varepsilon,\lambda}\right\Vert _{W^{1,p}(0,T;V)}+\left\Vert
u_{\varepsilon,\lambda}\right\Vert _{L^{m}(0,T;X)}  &  \leq C,\\
\left\Vert \eta_{\varepsilon,\lambda}^{1}\right\Vert _{L^{m^{\prime}%
}(0,T;X^{\ast})}  &  \leq C,\\
\left\Vert \xi_{\varepsilon,\lambda}\right\Vert _{L^{p^{\prime}}(0,T;V^{\ast
})}  &  \leq C\text{,}\label{est xi lambda}\\
\left\Vert f(u_{\varepsilon,\lambda})\right\Vert _{L^{p^{\prime}}(0,T;V^{\ast
})}  &  \leq C\text{.}%
\end{align}
 In particular, 
$$
\left\Vert u_{\varepsilon,\lambda}\right\Vert _{C([0,T];V)}  \leq C\text{.}
$$ 
Substituting $t=T$ into  estimate  (\ref{estimate * 2}) and recalling $\xi_{\varepsilon
,\lambda}(T)=0$, we get
\[
\varphi^{1}(u_{\varepsilon,\lambda}(T))\leq C\text{,}%
\]
which along with  assumptions  \textbf{(A3)}, \textbf{(A4)}, and \textbf{(A6)} implies%
\begin{align*}
|u_{\varepsilon,\lambda}(T)|_{X} \leq C, \quad
|\eta_{\varepsilon,\lambda}^{1}(T)|_{X^{\ast}} \leq C, \quad
|\eta_{\varepsilon,\lambda}^{2}(T)|_{V^{\ast}}  \leq C\text{.}%
\end{align*}
By using assumption \textbf{(A6)} again, we get
\begin{align}
\left\Vert \eta_{\varepsilon,\lambda}^{2}\right\Vert _{L^{p^{\prime}%
}(0,T;V^{\ast})}  &  \leq C\text{, }\label{estimate eta 2}\\
\int_{0}^{T}\varphi^{2}(u_{\varepsilon,\lambda})  &  \leq C.
\label{estimate varphi 2}%
\end{align}
Finally, by comparison in  equation  (\ref{regularized nonconv euler eq}),
\[
\left\Vert \varepsilon\xi_{\varepsilon,\lambda}^{\prime}\right\Vert
_{L^{p^{\prime}}(0,T;V^{\ast})+L^{m^{\prime}}(0,T;X^{\ast})}\leq C\text{.}%
\]

\subsection{The passage to the limit as $\lambda\rightarrow0$.}\label{Ss:lim}

Owing to the obtained uniform estimates, up to some (not relabeled)
subsequence $\lambda\rightarrow0$, we have
\begin{align*}
\eta_{\varepsilon,\lambda}^{1}  &  \rightarrow\eta_{\varepsilon}^{1}\text{
weakly in }L^{m^{\prime}}(0,T;X^{\ast}),\\
\eta_{\varepsilon,\lambda}^{2}  &  \rightarrow\eta_{\varepsilon}^{2}\text{
weakly in }L^{p^{\prime}}(0,T;V^{\ast}),\\
\xi_{\varepsilon,\lambda}  &  \rightarrow\xi_{\varepsilon}\text{ weakly in
}L^{p^{\prime}}(0,T;V^{\ast}),\\
\xi_{\varepsilon,\lambda}^{\prime}  &  \rightarrow\xi_{\varepsilon}^{\prime
}\text{ weakly in }L^{p^{\prime}}(0,T;V^{\ast})+L^{m^{\prime}}(0,T;X^{\ast
}),\\
u_{\varepsilon,\lambda}  &  \rightarrow u_{\varepsilon}\text{ weakly in
}W^{1,p}(0,T;V)\text{ and in }L^{m}(0,T;X)\text{, }\\
u_{\varepsilon,\lambda}(T)  &  \rightarrow v_{\varepsilon}\text{ weakly in
}X\text{,}\\
\eta_{\varepsilon,\lambda}^{2}(T)  &  \rightarrow q_{\varepsilon}^{2}\text{
weakly in }V^{\ast}\text{,}\\
f(u_{\varepsilon,\lambda})  &  \rightarrow f_{\varepsilon}^{\ast}\text{ weakly
in }L^{p^{\prime}}(0,T;V^{\ast})\text{,}%
\end{align*}
and hence,  thanks to Aubin-Lions-Simon's compactness lemma  (see \cite{Si}),%
\begin{align}
u_{\varepsilon,\lambda}  &  \rightarrow u_{\varepsilon}\text{ strongly in
}C([0,T];V)\text{,}\label{strong conv u lambda}\\
\xi_{\varepsilon,\lambda}  &  \rightarrow\xi_{\varepsilon}\text{ strongly in
}C([0,T];X^{\ast})\text{.} \label{strong c xi}%
\end{align}
In particular,%
\[
u_{\varepsilon,\lambda}(t)\rightarrow u_{\varepsilon}(t)\text{ strongly in
}V\text{ for all }t\in\lbrack0,T],
\]
 which also yields  $v_{\varepsilon}=u_{\varepsilon}(T)$, $u_{\varepsilon}\left(  0\right)
=u_{0}$, and $\xi_{\varepsilon}(T)=0$.  By virtue of   the continuity of $f:V \to V^*$ and the convergence (\ref{strong conv u lambda}) (see also Remark \ref{remark assumptions f}),  we get
$f_{\varepsilon}^{\ast}=f(u_{\varepsilon})$ and
\begin{equation}
f(u_{\varepsilon,\lambda})\rightarrow f(u_{\varepsilon})\text{ strongly in
}L^{p^{\prime}}(0,T;V^{\ast})\text{.} \label{conv f}%
\end{equation}
Thus,  we assure that  %
\begin{equation}
-\varepsilon\xi_{\varepsilon}^{^{\prime}}+\xi_{\varepsilon}+\eta_{\varepsilon
}^{1}-\eta_{\varepsilon}^{2}-f\left(  u_{\varepsilon}\right)  =0.
\label{eq epsilon}%
\end{equation}
The inclusions $\eta_{\varepsilon}^{2}\in\partial_{V}\varphi^{2}%
(u_{\varepsilon})$ and $q_{\varepsilon}^{2}\in\partial_{V}\varphi^{2}%
(u_{\varepsilon}(T))$ follow by a standard monotonicity argument (see, e.g.
\cite[Chap. II, Section 1.2]{Ba}) as a consequence of the strong convergence
(\ref{strong conv u lambda}). We shall now identify the limit $\eta
_{\varepsilon}^{1}$ as $\eta_{\varepsilon}^{1}\in\partial_{X}\varphi_{X}^{1}(u_{\varepsilon})$
a.e. in $(0,T)$. By a standard argument for monotone operators, it follows from the weak convergences obtained above that
\begin{align}
\liminf_{\lambda\rightarrow0}\int_{s}^{t}\left\langle \eta_{\varepsilon
,\lambda}^{1},u_{\varepsilon,\lambda}\right\rangle _{X}  &  \geq\int_{s}%
^{t}\left\langle \eta_{\varepsilon}^{1},u_{\varepsilon}\right\rangle
_{X}\text{,}\label{limimf eta 1}\\
\liminf_{\lambda\rightarrow0}\int_{s}^{t}\left\langle \xi_{\varepsilon
,\lambda},u_{\varepsilon,\lambda}^{\prime}\right\rangle _{V}  &  \geq\int
_{s}^{t}\left\langle \xi_{\varepsilon},u_{\varepsilon}^{\prime}\right\rangle
_{V}\text{} \label{liminf xi}%
\end{align}
for all $0\leq s\leq t\leq T$. Let us note that the quantity $a(t)=\liminf
_{\lambda\rightarrow0}|\xi_{\varepsilon,\lambda}(t)|_{V^{\ast}}^{p^{\prime}}%
$, belongs to the space $L^{1}(0,T)$ by Fatou's Lemma and  estimate
(\ref{est xi lambda}). In particular, $a(t)<\infty$ for a.a. $t\in(0,T)$, and
for such $t$, we can take a subsequence $\lambda_{n}^{t}$ (possibly depending
on $t$) such that
\[
\xi_{\varepsilon,\lambda_{n}^{t}}(t)\rightarrow\xi_{\varepsilon}(t)\text{
weakly in }V^{\ast}\text{.}%
\]
Thus, thanks to convergence (\ref{strong conv u lambda}), we observe that the set $\mathcal{L}\subset(0,T)$ defined by
\begin{align*}
\mathcal{L}  &  :=\left\{  t\in(0,T)\text{: }t\text{ is a Lebesgue point for
}t\longmapsto\left\langle \xi_{\varepsilon}(t),u_{\varepsilon}(t)\right\rangle
_{V}\text{,}\right. \\
&  \left.  \text{ and for any sequence }\lambda_{n}\rightarrow0\text{
there exists a subsequence }\lambda_{n}^{\prime}\right. \\
&  \left.  \text{ such that }\left\langle \xi_{\varepsilon,\lambda_{n}%
^{\prime}}(t),u_{\varepsilon,\lambda_{n}^{\prime}}(t)\right\rangle
_{V}\rightarrow\left\langle \xi_{\varepsilon}(t),u_{\varepsilon}%
(t)\right\rangle _{V}\text{ }\right\}  \text{}%
\end{align*}
has the full (Lebesgue) measure. Thus, by virtue of the convergences above and of (\ref{liminf xi}), we deduce, for arbitrary $t_{1},t_{2}\in\mathcal{L}$,
$t_{2}>t_{1}$ and a (not-relabeled) subsequence $\lambda\rightarrow 0$ (possibly depending on the choice of 
$t_{1},t_{2}$), that
\begin{align*}
& \limsup_{\lambda\rightarrow0}\int_{t_{1}}^{t_{2}}\left\langle \eta
_{\varepsilon,\lambda}^{1},u_{\varepsilon,\lambda}\right\rangle _{X} 
\\ & =\limsup_{\lambda\rightarrow0}\left\{  -\int_{t_{1}}^{t_{2}}\left\langle
\xi_{\varepsilon,\lambda},u_{\varepsilon,\lambda}\right\rangle _{V}%
+\int_{t_{1}}^{t_{2}}\left\langle \varepsilon\xi_{\varepsilon,\lambda}%
^{\prime},u_{\varepsilon,\lambda}\right\rangle _{X}\right. \\
&  \quad \left.  +\int_{t_{1}}^{t_{2}}\left\langle f(u_{\varepsilon,\lambda
}),u_{\varepsilon,\lambda}\right\rangle _{V}+\int_{t_{1}}^{t_{2}}\left\langle
\eta_{\varepsilon,\lambda}^{2},u_{\varepsilon,\lambda}\right\rangle
_{V}\right\} \\
&  \leq\lim_{\lambda\rightarrow0}\left\{  -\int_{t_{1}}^{t_{2}}\left\langle
\xi_{\varepsilon,\lambda},u_{\varepsilon,\lambda}\right\rangle _{V}%
+\int_{t_{1}}^{t_{2}}\left\langle f(u_{\varepsilon,\lambda}),u_{\varepsilon
,\lambda}\right\rangle _{V}+\int_{t_{1}}^{t_{2}}\left\langle \eta
_{\varepsilon,\lambda}^{2},u_{\varepsilon,\lambda}\right\rangle _{V}\right\}
\\
&  \quad +\lim_{\lambda\rightarrow0}\left.  \left\langle \varepsilon\xi
_{\varepsilon,\lambda},u_{\varepsilon,\lambda}\right\rangle _{V}\right\vert
_{t_{1}}^{t_{2}}-\liminf_{\lambda\rightarrow0}\int_{t_{1}}^{t_{2}}\left\langle
\varepsilon\xi_{\varepsilon\,\lambda},u_{\varepsilon,\lambda}^{\prime
}\right\rangle _{V}\\
&  \leq-\int_{t_{1}}^{t_{2}}\left\langle \xi_{\varepsilon},u_{\varepsilon
}\right\rangle _{V}+\int_{t_{1}}^{t_{2}}\left\langle f(u_{\varepsilon
}),u_{\varepsilon}\right\rangle _{V}+\int_{t_{1}}^{t_{2}}\left\langle
\eta_{\varepsilon}^{2},u_{\varepsilon}\right\rangle _{V}\\
&  \quad +\left.  \left\langle \varepsilon\xi_{\varepsilon},u_{\varepsilon
}\right\rangle _{V}\right\vert _{t_{1}}^{t_{2}}-\int_{t_{1}}^{t_{2}%
}\left\langle \varepsilon\xi_{\varepsilon},u_{\varepsilon}^{\prime
}\right\rangle _{V}\\
&  \leq\int_{t_{1}}^{t_{2}}\left\langle \eta_{\varepsilon}^{1},u_{\varepsilon
}\right\rangle _{X}\text{.}%
\end{align*}
Here, we  also  used the integration-by-parts formula%
\begin{equation}\label{CL}
\left.  \left\langle \varepsilon\xi_{\varepsilon},u_{\varepsilon}\right\rangle
_{V}\right\vert _{t_{1}}^{t_{2}}-\int_{t_{1}}^{t_{2}}\left\langle
\varepsilon\xi_{\varepsilon},u_{\varepsilon}^{\prime}\right\rangle _{V}%
=\int_{t_{1}}^{t_{2}}\left\langle \varepsilon\xi_{\varepsilon}^{\prime
},u_{\varepsilon}\right\rangle _{X}%
\end{equation}
derived in \cite{Ak-St2}. This  fact  together with (\ref{limimf eta 1}) yields%
\[
\lim_{\lambda\rightarrow0}\int_{s}^{t}\left\langle \eta_{\varepsilon,\lambda
}^{1},u_{\varepsilon,\lambda}\right\rangle _{X}=\int_{s}^{t}\left\langle
\eta_{\varepsilon}^{1},u_{\varepsilon}\right\rangle _{X}\text{ for a.e. }0\leq
s\leq t\leq T\text{,}%
\]
and hence, $\eta_{\varepsilon}^{1}\in\partial_{X}\varphi_{X}^{1}%
(u_{\varepsilon})$ a.e in $(0,T)$.

Let us next show the inclusion $\xi=\mathrm{d}_{V}\psi\left(  u^{\prime}\right)
$. For arbitrary $t_{1}$,$t_{2}\in\mathcal{L}$, $t_{2}>t_{1}$, by integration
by parts, we compute, up to a (not-relabeled) subsequence, 
\begin{align*}
&\limsup_{\lambda\rightarrow0}\int_{t_{1}}^{t_{2}}\left\langle \xi
_{\varepsilon,\lambda},u_{\varepsilon,\lambda}^{\prime}\right\rangle _{V}  \\
&=\lim_{\lambda\rightarrow0}\left.  \left\langle \xi_{\varepsilon,\lambda
},u_{\varepsilon,\lambda}\right\rangle _{V}\right\vert _{t_{1}}^{t_{2}%
}-\liminf_{\lambda\rightarrow0}\int_{t_{1}}^{t_{2}}  \left\langle \xi
_{\varepsilon,\lambda}^{\prime},u_{\varepsilon,\lambda}\right\rangle_{X} \\
&  =\left.  \left\langle \xi_{\varepsilon},u_{\varepsilon}\right\rangle
_{V}\right\vert _{t_{1}}^{t_{2}}-\frac{1}{\varepsilon}\lim_{\lambda
\rightarrow0}\int_{t_{1}}^{t_{2}}\left\langle \xi_{\varepsilon,\lambda}%
-\eta_{\varepsilon,\lambda}^{2}-f\left(  u_{\varepsilon,\lambda}\right)
,u_{\varepsilon,\lambda}\right\rangle _{V}\\
& \quad -\frac{1}{\varepsilon}\lim_{\lambda\rightarrow0}\int_{t_{1}}^{t_{2}%
}\left\langle \eta_{\varepsilon,\lambda}^{1},u_{\varepsilon,\lambda
}\right\rangle _{X}\text{.}%
\end{align*}
As a consequence of the convergence obtained above and of identity
(\ref{eq epsilon}), we get
\begin{align*}
&\limsup_{\lambda\rightarrow0}\int_{t_{1}}^{t_{2}}\left\langle \xi
_{\varepsilon,\lambda},u_{\varepsilon,\lambda}^{\prime}\right\rangle _{V}  \\
&=\left.  \left\langle \xi_{\varepsilon},u_{\varepsilon}\right\rangle
_{V}\right\vert _{t_{1}}^{t_{2}}-\frac{1}{\varepsilon}\int_{t_{1}}^{t_{2}%
}\left\langle \xi_{\varepsilon}-\eta_{\varepsilon}^{2}-f\left(  u_{\varepsilon
}\right)  ,u_{\varepsilon}\right\rangle _{V}-\frac{1}{\varepsilon}\int_{t_{1}%
}^{t_{2}}\left\langle \eta_{\varepsilon}^{1},u_{\varepsilon}\right\rangle
_{X}\\
&  =\left.  \left\langle \xi_{\varepsilon},u_{\varepsilon}\right\rangle
_{V}\right\vert _{t_{1}}^{t_{2}}-\int_{t_{1}}^{t_{2}}\left\langle
\xi_{\varepsilon}^{\prime},u_{\varepsilon}\right\rangle_{X} \stackrel{\eqref{CL}}= \int_{t_{1}%
}^{t_{2}}\left\langle \xi_{\varepsilon},u_{\varepsilon}^{\prime}\right\rangle
_{V}\text{.}%
\end{align*}
In particular, by virtue of (\ref{liminf xi}),
\[
\lim_{\lambda\rightarrow0}\int_{t_{1}}^{t_{2}}\left\langle \xi_{\varepsilon
,\lambda},u_{\varepsilon,\lambda}^{\prime}\right\rangle _{V}=\int_{t_{1}%
}^{t_{2}}\left\langle \xi_{\varepsilon},u_{\varepsilon}^{\prime}\right\rangle
_{V}\text{.}%
\]
From the arbitrariness of $t_{1},t_{2} \in {\mathcal L}$, we conclude that $\xi_{\varepsilon} =\mathrm{d}_{V}\psi\left(  u_{\varepsilon}^{\prime}\right)  $ a.e. in
$\left(  0,T\right)  $. This proves the first half of Theorem
\ref{full gen theorem}.

Before moving  on  to the causal limit as $\varepsilon\rightarrow0$, let us derive an energy inequality for later use. 
Repeating the same argument  as in (\ref{formula}), we have
\begin{align*}
\int_{0}^{t}\left\langle \xi_{\varepsilon,\lambda},u_{\varepsilon,\lambda
}^{\prime}\right\rangle _{V}  &  \leq\varepsilon\left\langle \xi
_{\varepsilon,\lambda}(t),u_{\varepsilon,\lambda}^{\prime}(t)\right\rangle
_{V}+\varepsilon\psi(0)\\
&  -\varphi^{1}(u_{\varepsilon,\lambda}(t))+\varphi^{1}(u_{0})+\int_{0}%
^{t}\left\langle \eta_{\varepsilon,\lambda}^{2}+f(u_{\varepsilon,\lambda
}),u_{\varepsilon,\lambda}^{\prime}\right\rangle _{V}\text{.}%
\end{align*}
Thus,  due to the convergences obtained above along with the lower 
semicontinuity of $\varphi^{1}$ and identity (\ref{eq epsilon}), we obtain 
 (recalling that $\xi_{\varepsilon, \lambda}(T)=0$) 
 \begin{align}
\int_{0}^{T}\left\langle \xi_{\varepsilon},u_{\varepsilon}^{\prime
}\right\rangle _{V}  &  \leq\liminf_{\lambda\rightarrow0}\int_{0}%
^{T}\left\langle \xi_{\varepsilon,\lambda},u_{\varepsilon,\lambda}^{\prime
}\right\rangle _{V} \leq\limsup_{\lambda\rightarrow0}\int_{0}^{T}\left\langle \xi_{\varepsilon,\lambda},u_{\varepsilon,\lambda}^{\prime}\right\rangle _{V} \label{bbb}\\%
&\leq\varepsilon\psi(0)-\liminf_{\lambda\rightarrow0}\varphi^{1}(u_{\varepsilon
,\lambda}(T))+\varphi^{1}(u_{0})\nonumber\\
&\quad  +\limsup_{\lambda\rightarrow0}\varphi_{\lambda}^{2}(u_{\varepsilon,\lambda
}(T))-\varphi^{2}(u_{0})+\lim_{\lambda\rightarrow0}\int_{0}^{T}\left\langle
f(u_{\varepsilon,\lambda}),u_{\varepsilon,\lambda}^{\prime}\right\rangle
_{V}.\nonumber
 \end{align}
Note that, by definition of subdifferential and by using the convergences obtained above, we have,  as  $\lambda\rightarrow0$,
\[
\varphi^{2}_{\lambda } (u_{\varepsilon,\lambda}(T))
\leq\varphi^{2}_{\lambda }(u_{\varepsilon
}(T))+\left\langle \eta_{\varepsilon,\lambda}^{2}(T),u_{\varepsilon,\lambda
}(T)-u_{\varepsilon}(T)\right\rangle _{V}\rightarrow\varphi^{2}(u_{\varepsilon
}(T))%
\]
(we also used the fact that $u_{\vep,\lambda}(T) \in D(\varphi^1) \subset D(\varphi^2)$ by (A6)), i.e.,%
\[
\limsup_{\lambda\rightarrow0}\varphi^{2}_{\lambda }(u_{\varepsilon,\lambda}%
(T))\leq\varphi^{2}(u_{\varepsilon}(T))\text{.}%
\]
Thus, substituting the above into (\ref{bbb}), we get
 \begin{align}
  \label{form3}
\int_{0}^{T}\left\langle \xi_{\varepsilon},u_{\varepsilon}^{\prime }\right\rangle _{V} &\leq\varepsilon\psi(0)-\varphi^{1}(u_{\varepsilon
  }(T))\\
  &\quad +\varphi^{1}(u_{0})+\varphi^{2}(u_{\varepsilon}(T))-\varphi^{2}%
(u_{0})+\int_{0}^{T}\left\langle f(u_{\varepsilon}),u_{\varepsilon}^{\prime
}\right\rangle _{V}\text{.} \nonumber
 \end{align}

\subsection{The causal limit as $\varepsilon\rightarrow0$.}

We now deal with the passage to the limit  as  $\varepsilon\rightarrow0$. We first note that, thanks to the estimates mentioned above and the weak lower
semicontinuity of norms, we have%
\begin{align*}
\left\Vert u_{\varepsilon}\right\Vert _{W^{1,p}(0,T;V)}+\left\Vert
u_{\varepsilon}\right\Vert _{L^{m}(0,T;X)}  &  \leq C,\\
\left\Vert \eta_{\varepsilon}^{1}\right\Vert _{L^{m^{\prime}}(0,T;X^{\ast})}
&  \leq C,\\
\left\Vert \xi_{\varepsilon}\right\Vert _{L^{p^{\prime}}(0,T;V^{\ast})}  &
\leq C\text{,}\\
\left\Vert f(u_{\varepsilon})\right\Vert _{L^{p^{\prime}}(0,T;V^{\ast})}  &
\leq C,\\
\left\Vert \eta_{\varepsilon}^{2}\right\Vert _{L^{p^{\prime}}(0,T;V^{\ast})}
&  \leq C\text{,}\\
\left\Vert \varepsilon\xi_{\varepsilon}^{\prime}\right\Vert _{L^{p^{\prime}%
}(0,T;V^{\ast})+L^{m^{\prime}}(0,T;X^{\ast})}  &  \leq C,\\
\left\Vert u_{\varepsilon}(T)\right\Vert _{X}  &  \leq C,\\
\left\Vert q_{\varepsilon}^{2}\right\Vert _{V^{\ast}}  &  \leq C,
\end{align*}
where $q^2_\vep \in \partial_V \varphi^2(u_\vep(T))$ (see \S \ref{Ss:lim}). 
Up to a (not relabeled) subsequence, we get the following convergence results
\begin{alignat*}{4}
\eta_{\varepsilon}^{1}  &  \rightarrow\eta^{1}&&\text{ weakly in }L^{m^{\prime}%
}(0,T;X^{\ast}),\\
\eta_{\varepsilon}^{2}  &  \rightarrow\eta^{2}&&\text{ weakly in }L^{p^{\prime}%
}(0,T;V^{\ast}),\\
\xi_{\varepsilon}  &  \rightarrow\xi&&\text{ weakly in }L^{p^{\prime}%
}(0,T;V^{\ast}),\\
\varepsilon\xi_{\varepsilon}^{\prime}  &  \rightarrow0&&\text{ weakly in
}L^{p^{\prime}}(0,T;V^{\ast})+L^{m^{\prime}}(0,T;X^{\ast}),\\
\varepsilon\xi_{\varepsilon}  & \rightarrow0 && \text{ strongly in } C([0,T];X^*), \\
u_{\varepsilon}  &  \rightarrow u&&\text{ weakly in }W^{1,p}(0,T;V)\text{ and in
}L^{m}(0,T;X)\text{, }\\
& &&\text{ strongly in }C\left(  [0,T];V\right)\text{,}\\
 u_{\varepsilon}(T)  &  \rightarrow v&&\text{ weakly in }X\text{,}\\
q_{\varepsilon}^{2}  &  \rightarrow q^{2}&&\text{ weakly in }V^{\ast}\text{,}\\
f(u_{\varepsilon})  &  \rightarrow f(u)\ &&\text{ strongly in }L^{p^{\prime}%
}(0,T;V^{\ast}).
\end{alignat*}
In particular, $v=u(T)$, and $u\left(  0\right)  =u_{0}$, and moreover,
\begin{equation}
\xi+\eta^{1}-\eta^{2}-f\left(  u\right)  =0 \ \mbox{ in } X^*, \quad 0 < t < T. \label{final eq}%
\end{equation}
As a consequence of the strong convergence $u_{\varepsilon}\rightarrow u$ in $C([0,T];V)$ and the weak convergence of $\eta^2_\vep$ (respectively, $q^2_\vep$), by the demiclosedness of maximal monotone operators (see also~\cite[Proposition 1.1]{Ke}), we obtain the relation $\eta^{2}(t)\in\partial_{V}\varphi^{2}(u(t))$ for a.e.~$t\in (0,T)$ (respectively, $q^{2}\in\partial_{V}\varphi^{2}(u(T))$).  Repeating the same argument  as in Section \ref{convex case}, we can obtain
\[
\limsup_{\varepsilon\rightarrow0}\int_{0}^{T}\left\langle \eta_{\varepsilon
}^{1},u_{\varepsilon}\right\rangle _{X}\leq\int_{0}^{T}\left\langle \eta
^{1},u\right\rangle _{X}\text{,}%
\]
which proves $\eta^{1}\in\partial_{X}\varphi_{X}^{1}(u)$, and hence, $\eta
^{1}\in\partial_{V}\varphi^{1}(u)$ by $\eta^1 = f(u) + \eta^2 - \xi \in L^{p'}(0,T;V^*)$ (see~\cite{Ak-St2}). In order to  prove the relation $\xi=\mathrm{d}_{V}\psi(u^{\prime})$, we proceed as follows: By definition of
subdifferential and by using the convergences obtained so far, we have, as
$\varepsilon\rightarrow0$,
\[
\varphi^{2}(u_{\varepsilon}(T))\leq\varphi^{2}(u(T))+\left\langle
q_{\varepsilon}^{2},u_{\varepsilon}(T)-u(T)\right\rangle _{V}\rightarrow
\varphi^{2}(u(T))\text{,}%
\]
i.e.,%
\[
\limsup_{\varepsilon\rightarrow0}\varphi^{2}(u_{\varepsilon}(T))\leq
\varphi^{2}(u(T))\text{.}%
\]
Thus, by the convergences above along with the semicontinuity of $\varphi^{1}$, estimate (\ref{form3}) and identity (\ref{final eq}), we derive 
\begin{align*}
\limsup_{\varepsilon\rightarrow0}\int_{0}^{T}\left\langle \xi_{\varepsilon
},u_{\varepsilon}^{\prime}\right\rangle _{V}  &  \leq\lim_{\varepsilon
\rightarrow0}\varepsilon\psi(0)-\liminf_{\varepsilon\rightarrow0}\varphi
^{1}(u_{\varepsilon}(T))+\varphi^{1}(u_{0})\\
&  +\limsup_{\varepsilon\rightarrow0}\varphi^{2}(u_{\varepsilon}%
(T))-\varphi^{2}(u_{0})+\lim_{\varepsilon\rightarrow0}\int_{0}^{T}\left\langle
f(u_{\varepsilon}),u_{\varepsilon}^{\prime}\right\rangle _{V}\\
&  \leq-\varphi^{1}(u(T))+\varphi^{1}(u_{0})+\varphi^{2}(u(T))-\varphi
^{2}(u_{0})+\int_{0}^{T}\left\langle f(u),u^{\prime}\right\rangle _{V}\\
&  =\int_{0}^{T}\left\langle \xi,u^{\prime}\right\rangle _{V}\text{.}%
\end{align*}
Here we also used the chain rule developed in~\cite{Ak-St2} (cf.~\eqref{CL}). 
Thus, $\xi(t)=\mathrm{d}_{V}\psi(u^{\prime}(t))$ for a.a. $t\in(0,T)$.

 \begin{remark}
  {\rm
Thanks to results of Section \ref{causal limit} and to uniform
estimates obtained in Section \ref{nonconvex energies}, one can pass to
the limit also in the opposite order: first let $\varepsilon\rightarrow0$ and then let $\lambda\rightarrow0$. 
  }
 \end{remark}

\section{WED principle for  doubly-nonlinear  flows of nonconvex
energies\label{section wed minimization}}

In this short section, as a by-product, we develop a variational characterization based on WED functionals for  doubly-nonlinear  flows of nonconvex energy functionals (i.e., the case $f=0$). The WED variational principle has been mainly studied for convex (at most $\lambda$-convex) energy functionals; on the other hand, nonconvex energy functionals have not yet been treated except in~\cite{Ak-St}, where standard gradient flows (i.e., $\psi$ is the quadratic potential) of non($\lambda$-)convex energy functionals are treated. In particular,  doubly-nonlinear  flows with nonconvex energies have never been studied so far.  The following theorem provides a generalization of the results in \cite{Ak-St} to  doubly-nonlinear  flows of nonconvex energies.

\begin{theorem}
 Assume \emph{\textbf{(A1)}-\textbf{(A6)}} and let $f=0$.  Then, for every $\varepsilon>0$, the WED functional $W_{\varepsilon}$ defined
by 
\[
W_{\varepsilon}(u)=\left\{
\begin{array}
[c]{cl}%
\int_{0}^{T}\mathrm{e}^{-t/\varepsilon}\left(  \varepsilon\psi(u^{\prime
})+\varphi^{1}(u)-\varphi^{2}(u)\right)  & \text{if } \ u\in
K(u_{0})\cap L^{m}\left(  0,T;X\right)  \text{,}\\
\infty & \text{ otherwise,}%
\end{array}
\right.
\]
admits at least one global minimizer over the set $K(u_{0}):=\{u\in W^{1,p}\left(
0,T;V\right)  :u(0)=u_{0}\}$.  Furthermore, every local minimizer $u_{\varepsilon}$
solves \emph{(\ref{nonconvex euler eq})-(\ref{nc el ic})}, and  moreover,  $u_{\varepsilon}\rightarrow u$ weakly in $W^{1,p}(0,T;V)\cap L^{m}\left(  0,T;X\right)$ and strongly in $C([0,T];V)$, where the limit  $u$ solves \emph{(\ref{nonconvex target equation})-(\ref{ic nc})}.
\end{theorem}

\begin{proof}
We shall follow the strategy for proving \cite[Theorem 4.1]{Ak-St}. For all $\lambda > 0$, let us define regularized WED functionals 
$W_{\varepsilon,\lambda}:L^{p}(0,T;V)\rightarrow%
\mathbb{R}
\cup\{\infty\}$ by
\[
W_{\varepsilon,\lambda}(u)=\left\{
\begin{array}
[c]{cl}%
\int_{0}^{T}\mathrm{e}^{-t/\varepsilon}\left(  \varepsilon\psi(u^{\prime
})+\varphi^{1}(u)-\varphi_{\lambda}^{2}(u)\right)  & \text{if } \ u\in
K(u_{0})\cap L^{m}\left(  0,T;X\right)  \text{,}\\
\infty & \text{ otherwise, }%
\end{array}
\right.
\]
and decompose  them  as the difference $W_{\varepsilon,\lambda}(u) =C_{\varepsilon}^{1}(u)-C_{\varepsilon,\lambda}^{2}(u)$ of convex functionals
\begin{align*}
C_{\varepsilon}^{1}(u)  &  :=\left\{
\begin{array}
[c]{cl}%
\int_{0}^{T}\mathrm{e}^{-t/\varepsilon}\left(  \varepsilon\psi(u^{\prime
})+\varphi^{1}(u)\right)  & \text{if } \ u\in K(u_{0})\cap L^{m}\left(
0,T;X\right)  \text{,}\\
\infty & \text{ otherwise, }%
\end{array}
\right. \\
C_{\varepsilon,\lambda}^{2}(u)  &  :=\int_{0}^{T}\mathrm{e}^{-t/\varepsilon
}\varphi_{\lambda}^{2}(u)\text{,}%
\end{align*}
where $\varphi_{\lambda}^{2}$ denotes the Moreau-Yosida regularization of
$\varphi^{2}$ as in Section \ref{nonconvex energies}. Note that
$C_{\varepsilon}^{1}$ is lower semicontinuous in $L^{p}\left(  0,T;V\right)
$. Moreover, for all sequences $u_{n}\rightarrow u$ strongly in $L^{p}%
(0,T;V)$, we have $C_{\varepsilon,\lambda}^{2}(u_{n})\rightarrow
C_{\varepsilon,\lambda}^{2}(u)$. Indeed, by convexity and lower semicontinuity
of $\varphi_{\lambda}^{2}$, one has
\[
\liminf_{n\rightarrow\infty}C_{\varepsilon,\lambda}^{2}(u_{n})\geq
C_{\varepsilon,\lambda}^{2}(u)\text{.}%
\]
Furthermore, by definition of subdifferential, we get
\begin{align}
 \label{cont C2}
 \limsup_{n\rightarrow\infty}C_{\varepsilon,\lambda}^{2}(u_{n})  &  \leq
C_{\varepsilon,\lambda}^{2}(u)+\lim_{n\rightarrow\infty}\left\langle
\partial_{L^{p}(0,T;V)}C_{\varepsilon,\lambda}(u_{n}),u_{n}-u\right\rangle
_{L^{p}(0,T;V)}\\
&  = C_{\varepsilon,\lambda}^{2}(u)+\lim_{n\rightarrow\infty}\int_{0}%
^{T}\exp(-t/\varepsilon)\left\langle \partial_{V}\varphi_{\lambda}^{2}%
 (u_{n}),u_{n}-u\right\rangle _{V}\nonumber\\
 &=C_{\varepsilon,\lambda}^{2}(u)\nonumber
\end{align}
(see also \S \ref{yosida regularization} in Appendix). 
Thanks to (\ref{phi2 bound}) and the fact that
$\varphi_{\lambda}^{2}(u)\leq\varphi^{2}(u)$ for all $u\in D\left(
\varphi^{2}\right)  $, $W_{\varepsilon,\lambda}$ is bounded from below. Let
$\{u_{n}\}$ be a  minimizing  sequence of $W_{\varepsilon,\lambda}$. Thus, by using the coercivity assumptions (\ref{coercivity psi}), (\ref{coercivity phi}), and  assumption  (\ref{phi2 bound}) (see also \cite{Ak-St}), we can deduce that the sequence $\{u_{n}\}$ is uniformly bounded in $W^{1,p}(0,T;V)\cap L^{m}\left(  0,T;X\right)$, and up to a (not relabeled) subsequence, $u_{n}\rightarrow u_{\varepsilon,\lambda}$ strongly in $C([0,T];V)$ for some $u_{\vep,\lambda} \in W^{1,p}(0,T;V) \cap L^m(0,T;X)$. Thus, by using the lower semicontinuity of $C^{1}_{\varepsilon}$ and (\ref{cont C2}), we conclude that
$u_{\varepsilon,\lambda}$ is a minimizer of $W_{\varepsilon,\lambda}$. In
particular, $u_{\varepsilon,\lambda}$ solves
\[
0\in\partial_{L^{p}(0,T;V)}W_{\varepsilon,\lambda}\left(  u_{\varepsilon
,\lambda}\right)  =\partial_{L^{p}(0,T;V)}C_{\varepsilon}^{1}\left(
u_{\varepsilon,\lambda}\right)  -\mathrm{d}_{L^{p}(0,T;V)}C_{\varepsilon
,\lambda}^{2}\left(  u_{\varepsilon,\lambda}\right)  \text{,}%
\]
which is equivalent to (\ref{regularized nonconv euler eq})-(\ref{regularized nonconvex euler eq ic}) with $f = 0$ (see \cite{Ak-St2}). Thus, as in Section \ref{nonconvex energies}, we are ready to prove that (up to
subsequences) $u_{\varepsilon,\lambda}\rightarrow u_{\varepsilon}$ weakly in
$W^{1,p}(0,T;V)\cap L^{m}\left(  0,T;X\right)  $ for $\lambda\rightarrow0$ and
that the limit $u_{\varepsilon}$ solves (\ref{nonconvex euler eq})-(\ref{nc el ic}).
Furthermore, by lower semicontinuity, we have
\[
\liminf_{\lambda\rightarrow0}C_{\varepsilon}^{1}(u_{\varepsilon,\lambda})\geq
C_{\varepsilon}^{1}(u_{\varepsilon})\text{,}%
\]
and as in~\cite[Theorem 4.1]{Ak-St}, 
\[
\limsup_{\lambda\rightarrow0}C_{\varepsilon,\lambda}^{2}(u_{\varepsilon
,\lambda})\leq\limsup_{\lambda\rightarrow0}\int_{0}^{T}\mathrm{e}%
^{-t/\varepsilon}\varphi^{2}(u_{\varepsilon,\lambda})\leq\int_{0}%
^{T}\mathrm{e}^{-t/\varepsilon}\varphi^{2}(u_{\varepsilon})\text{.}%
\]
Moreover, $W_{\varepsilon,\lambda}(v)\rightarrow W_{\varepsilon}(v)$ for all
$v\in D(W_{\varepsilon})$  as $\lambda \to 0$.  Thus, $u_{\varepsilon}$ minimizes $W_{\varepsilon}%
$. Indeed, let $v\in D(W_{\varepsilon})$.  Then it follows that 
\[
W_{\varepsilon}(u_{\varepsilon}) =  C_{\varepsilon}^{1}(u_{\varepsilon}%
)-\int_{0}^{T}\mathrm{e}^{-t/\varepsilon}\varphi^{2}(u_{\varepsilon}%
)\leq\liminf_{\lambda\rightarrow0}W_{\varepsilon,\lambda}(u_{\varepsilon
,\lambda})\leq\lim_{\lambda\rightarrow0}W_{\varepsilon,\lambda}%
(v)=W_{\varepsilon}(v)\text{.}%
\]
This proves existence of a minimizer $u_{\varepsilon}$ of the WED
functional $W_{\varepsilon}$ such that $u_{\varepsilon}$ solves the
Euler-Lagrange equation (\ref{nonconvex euler eq})-(\ref{nc el ic}) (with $f=0$) of $W_\vep$ (or the elliptic-in-time regularized equation).

We  next  prove that every local minimizer of $W_{\varepsilon}$ solves
(\ref{nonconvex euler eq})-(\ref{nc el ic}) by using a penalization argument
(cf. \cite{Ak-St}). To this aim, let $\hat{u}_{\varepsilon}$ be a local 
minimizer of $W_{\varepsilon}$ and define the penalized functional,
$$
\hat{W}_{\varepsilon}(u) =\int_{0}^{T}\mathrm{e}^{-t/\varepsilon}\left(
\varepsilon\psi(u^{\prime})+\hat{\varphi}^{1}(u)-\varphi^{2}(u)\right)
 $$
 with 
 $$
\hat{\varphi}^{1}(u)  =\varphi^{1}(u)+\alpha \frac{1}{p}|u-\hat{u}_{\varepsilon
}|_{V}^{p}\text{}%
$$ 
for some $\alpha > 0$.
Then, arguing as in \cite[Theorem 4.2]{Ak-St} one can check that, for $\alpha = \alpha (\hat{u}_\varepsilon)$ sufficiently big, the functional $\hat{W}_{\varepsilon}$ admits a unique global minimizer $\hat{u}_\varepsilon$.
 Note that $\hat{\varphi}%
^{1}$ satisfies assumption \textbf{(A3)}, \textbf{(A4)},  and  \textbf{(A6)}. Thus, by applying the results obtained so far, for  each  $\lambda>0$, there exists a minimizer $\tilde{u}_{\varepsilon,\lambda}$ of the regularized WED functional,  
$$
\hat{W}_{\varepsilon,\lambda}(u)=\int_{0}^{T}\mathrm{e}^{-t/\varepsilon
}\left(  \varepsilon\psi(u^{\prime})+\hat{\varphi}^{1}(u)-\varphi_{\lambda
}^{2}(u)\right),
$$
such that $\tilde{u}_{\vep, \lambda}$ solves the Euler-Lagrange equation, 
\begin{gather*}
-\varepsilon\tilde{\xi}_{\varepsilon,\lambda}^{\prime}+\tilde{\xi
}_{\varepsilon,\lambda}+\tilde{\eta}_{\varepsilon,\lambda}^{1}+\alpha\tilde{\gamma}_{\varepsilon,\lambda} -\tilde{\eta
}_{\varepsilon,\lambda}^{2}   =0\text{
 in }X^{\ast}\text{ a.e. in }(0,T)\text{,}\\
 \tilde{\xi}_{\varepsilon,\lambda} = \mathrm{d}_{V}\psi(\tilde {u}_{\varepsilon,\lambda}^{\prime}), \quad
\tilde{\eta}_{\varepsilon,\lambda}^{1} \in\partial_{X}\varphi_{X}%
^{1}(\tilde{u}_{\varepsilon,\lambda}), \quad \tilde{\eta
}_{\varepsilon,\lambda}^{2} = \partial_{V}\varphi_{\lambda}^{2}(\tilde
{u}_{\varepsilon,\lambda})\text{,}\\
\tilde{\gamma}_{\varepsilon,\lambda}    =|\tilde{u}_{\varepsilon,\lambda
}-\hat{u}_{\varepsilon}|_{V}^{p-2}F_{V}\left(  \tilde{u}_{\varepsilon,\lambda
}-\hat{u}_{\varepsilon}\right)  \text{,}\\
\tilde{u}_{\varepsilon,\lambda}(0) =u_{0}\text{, }\quad\tilde{\xi
}_{\varepsilon,\lambda}(T)=0,
\end{gather*}
where $F_{V}:V\rightarrow V^{\ast}$ denotes the duality mapping between $V$ and $V^*$. Indeed, the functional $u \mapsto \int_{0}^{T}e^{-t/\varepsilon}\frac{1}{p}|u(t)-\hat {u}_{\varepsilon}(t)|_{V}^{p}$  is Fr\'{e}chet differentiable in $L^{p}(0,T;V)$ and its Fr\'{e}chet derivative at $\tilde{u}_{\varepsilon,\lambda}$ is given as $ e^{-t/\vep} \tilde\gamma_{\vep,\lambda}$. We can now derive uniform
estimates for $\tilde{u}_{\varepsilon,\lambda}$ and prove convergence of
$\tilde{u}_{\varepsilon,\lambda}$ to a limit $\tilde{u}_{\varepsilon}$ which
minimizes $\hat{W}_{\varepsilon}$ and solves%
\begin{gather*}
-\varepsilon\tilde{\xi}_{\varepsilon}^{\prime}+\tilde{\xi}_{\varepsilon
}+\tilde{\eta}_{\varepsilon}^{1}-\tilde{\eta}_{\varepsilon}^{2}+\alpha\tilde{\gamma
}_{\varepsilon}   =0\text{ in }X^{\ast}\text{ a.e. in }(0,T)\text{,}\\
\tilde{\xi}_{\varepsilon} = \mathrm{d}_{V}\psi(\tilde{u}_{\varepsilon
}^{\prime})\text{,} \quad \tilde{\eta}_{\varepsilon}^{1}   \in\partial_{X}\varphi_{X}^{1}(\tilde
{u}_{\varepsilon}), \quad \tilde{\eta}_{\varepsilon}^{2}%
\in \partial_{V}\varphi^{2}(\tilde{u}_{\varepsilon}) \text{,}\\
\tilde{\gamma}_{\varepsilon}  =|\tilde{u}_{\varepsilon}-\hat{u}%
_{\varepsilon}|_{V}^{p-2}F_{V}\left(  \tilde{u}_{\varepsilon}-\hat
{u}_{\varepsilon}\right)  \text{,}\\
\tilde{u}_{\varepsilon}(0)   =u_{0}, \quad \tilde{\xi
}_{\varepsilon}(T)=0.
\end{gather*}
Recall that, by penalization, the unique minimizer of $\hat{W}_{\varepsilon}$ is
$\hat{u}_{\varepsilon}$ (see also \cite{Ak-St3}). Thus, $\hat{u}_{\varepsilon
}=\tilde{u}_{\varepsilon}$ and, by the substitution of this relation into the
equation above, $\tilde{\gamma}_{\varepsilon}=0$, and hence, $\hat
{u}_{\varepsilon}$ solves (\ref{nonconvex euler eq})-(\ref{nc el ic}).

Finally, the  limiting procedure as  $\varepsilon\rightarrow0$ has already been proved in Section \ref{nonconvex energies}.
This completes the proof. 
\end{proof}

\section{An alternative approach}\label{second approach}

In this section, we exhibit a slightly different approach. More precisely,  we shall prove the assertion of  Theorem \ref{full gen theorem} under different assumptions on $f$ and $\varphi^{2}$.

We recall that, in Section \ref{convex case}, we introduced the map $S$ defined by
\begin{align*}
S  &  :L^{p}(0,T;V)\rightarrow L^{p}(0,T;V),\\
S  &  :v\mapsto w:=f(v) \longmapsto u,
\end{align*}
where $u$ is the global minimizer of $I_{\varepsilon,w}$, and looked for a fixed point of $S$  
to construct a solution of the elliptic-in-time regularization \eqref{fix point Euler}-\eqref{fix point ic}. Note that the map
$S$ is the composition of two different maps: $v\mapsto f(v)$ and
$w\longmapsto u := \mathrm{argmin} \, I_{\vep,w} $.  Alternatively, one may  consider a map $\tilde{S}$ which
is the composition of the same maps in the opposite order, namely
\begin{align*}
\tilde{S}  &  :L^{p^{\prime}}(0,T;V^{\ast})\rightarrow L^{p^{\prime}%
}(0,T;V^{\ast})\text{,}\\
\tilde{S}  &  :v\mapsto u\longmapsto f(u)\text{,}%
\end{align*}
where $u$ is the global minimizer of $I_{\varepsilon,v}$. Note that if
$\tilde{S}$ has a fixed point $v$, then the minimizer of $I_{\varepsilon,v}$
solves \eqref{fix point Euler}-\eqref{fix point ic}. In order to apply the Schaefer fixed-point theorem to $\tilde{S}$, one has to  check the (strong) continuity of $\tilde{S}$ in $L^{p'}(0,T;V^*)$ as in  \S \ref{sec S cont}.  Furthermore, as for  nonconvex  energies, $f$ is replaced by $f + \partial_V \varphi^2_\lambda$ above. In case $V$ is a Hilbert space, one can prove the (strong) continuity of $\tilde S$ by employing the Lipschitz continuity of Yosida approximations in the Hilbert space setting. On the other hand, in case $V$ is a Banach space, it seems somewhat difficult to prove the (strong) continuity of $\tilde S$, due to the lack of (strong) continuity of the Yosida approximation $\partial_{V}\varphi_{\lambda }^{2} : V \to V^*$ in a Banach space setting (it is only demicontinuous. See Appendix and \cite{Ba}). Hence in order to recover the continuity of $\tilde{S}$, we assume that
$$
\varphi^{2} \mbox{ is of class } C^1 (V;\mathbb{R}) 
$$
(in the sense of Fr\'{e}chet derivative). Then, $\partial_{V}\varphi^{2}$ is single-valued and continuous from $V$ into $V^*$. Hence, it is no longer necessary to employ the Yosida approximation of $\partial_V \varphi^2$. On the other hand, the growth and the continuity conditions in \textbf{(A5)} can be relaxed as follows:
\begin{description}
\item[(A5${}'$)] $f:X \to V^\ast$ satisfies 
\begin{equation}
|f(u)|_{V^\ast}^{p^{\prime}} \leq C (1+\varphi^1(u)+|u|_V^p) \text{ for all } u \in X \label{growth condition 2}
\end{equation}
and some positive constant C. Moreover, if $u \in L^{m}(0,T;X)\cap W^{1,p}(0,T;V)$, then $f(u) \in L^{p^{\prime}%
}(0,T;V^{\ast})$. Furthermore, if $u_{n} \to u$ weakly in $L^{m}(0,T;X)$ and in $W^{1,p}(0,T;V)$, and $(f(u_{n}))$ is bounded in $L^{p^{\prime}}(0,T;V^{\ast })$, then $f(u_{n})\rightarrow f(u)$ strongly in $L^{p^{\prime}}(0,T;V^{\ast})$.
\end{description}
Note that these assumptions guarantee the continuity of $\tilde{S}$ corresponding to  nonconvex  energies.
In particular, one may prove the following:

\begin{theorem} \label{thm second approach}
Let assumptions \emph{\textbf{(A1)}-\textbf{(A4)}}, \emph{\textbf{(A5${}'$)}}, and
\emph{\textbf{(A6) }}be satisfied. Assume either that $\varphi^{2}$ is  of class $C^1$ in $V$  or that $V$ is a Hilbert space. Then, there exists
$\varepsilon_{0}>0$ such that for each $\varepsilon\in(0,\varepsilon_{0})$
 the system \emph{(\ref{nonconvex euler eq})-(\ref{nc el ic})} admits a  solution $u_{\varepsilon}$. Moreover, there exists a sequence $\varepsilon
_{n}\rightarrow0$ such that $u_{\varepsilon_{n}}\rightarrow u$ weakly in
$L^{m}(0,T;X)\cap W^{1,p}\left(  0,T;V\right)  $ and strongly in $C\left(
[0,T];V\right)  $ and  the limit  $u$ solves \emph{(\ref{nonconvex target equation})-(\ref{ic nc}).}
\end{theorem}

\begin{proof}
Let us start with the case $\varphi^2=0$, i.e., $\phi=\varphi^1$. Our proof follows the scheme of the proof of Theorem \ref{thm convex} given in Section
\ref{convex case}. Let $\tilde{S}: L^{p^\prime}(0,T;V^\ast) \to L^{p^\prime}(0,T;V^\ast)$ be defined as above. Fix $v \in L^{p^\prime}(0,T;V^\ast)$ and let $u$ be the global minimizer of $I_{\varepsilon,v}$. Then, $f(u)=\tilde{S}(v) \in L^{p^\prime}(0,T;V^\ast)$. Following the procedure of \S \ref{Sss:3.1.1}, one can prove that $u$ is uniformly bounded in $W^{1,p}(0,T;V) \cap L^m(0,T;X)$ by a constant depending on $\Vert v \Vert_{L^{p^\prime}(0,T;V^\ast)}$. This fact together with the (weak-strong) continuity assumption on $f$ in \textbf{(A5')} ensures both continuity and compactness of the map $\tilde{S}$.
Let now $v \in L^{p^\prime}(0,T;V^\ast)$, $\alpha \in (0,1]$, and $u=\operatorname{argmin} I_{\varepsilon,v}$ be such that 
$$
v=\alpha \tilde{S}(v) (=\alpha f(u)). 
$$
Then, $u$ solves system (\ref{Eu})-(\ref{Eu3}) with $f(u)$ replaced by $\alpha f(u)$. By testing equation (\ref{Eu}) with $u^\prime $ and proceeding as for (\ref{est2})-(\ref{estimate 1}), we get, for $t \in (0,T]$
\begin{align*}
& \int_0^T|u^\prime|^p_V \leq C+C\int_0^T |\alpha f(u)|_{V^\ast}^{p^\prime}, \\
& |u(t)|_V^p + \phi(u(t)) \leq C+C \int_0^t |\alpha f(u)|_{V^\ast}^{p^\prime} + C \int_0^t |u|_V^p + \varepsilon \psi^\ast(\xi(t)). 
\end{align*}
Using now the growth assumption on $f$ in \textbf{(A5')}, we obtain 
\begin{align}
& \int_0^T|u^\prime|^p_V \leq C+C\int_0^T \left( |u|_{V}^{p} + \phi(u)  \right), \label{eee1} \\
& |u(t)|_V^p + \phi(u(t)) \leq C+C \int_0^t \left( |u|_{V}^{p} + \phi(u)  \right) + \varepsilon \psi^\ast(\xi(t)). \label{eee2} 
\end{align}
Thanks to Gronwall's Lemma, we get 
$$
|u(t)|_V^p + \phi(u(t)) \leq C + \varepsilon \psi^\ast(\xi(t)) + C \varepsilon \int_0^t \varepsilon \psi^\ast(\xi).
$$
By substituting the latter into (\ref{eee2}), integrating both sides over $[0,T]$, taking the sum with (\ref{eee1}), and substituting again the latter into the right hand side,  we obtain
$$
\int_0^T|u^\prime|^p_V + \int_0^T |u|_V^p +\int_0^T \phi(u) \leq C + C \varepsilon \int_0^t \varepsilon \psi^\ast(\xi)
\leq C+C\varepsilon \int_0^T |u^\prime|^p_V.
$$
Here we used estimate (\ref{fenchel estimate}). For $\varepsilon$ sufficiently small this yields a bound of $$\Vert u \Vert_{W^{1,p}(0,T;V) \cap L^m(0,T;X)}$$ uniform in $\varepsilon$ and $\alpha$ and hence an uniform bound on$$\Vert f(u) \Vert_{L^{p^\prime}(0,T;V^\ast)}$$ thanks to \textbf{(A5')}. Thus, the set $\left\{ v \in L^{p^\prime}(0,T;V^\ast): v=\alpha \tilde{S}(v) \text{ for some } \alpha \in [0,1] \right\}$ is bounded. The Schaefer's fixed-point theorem allows us to conclude that $\tilde{S}$ has a fixed point $v$. As a consequence $u = \operatorname{argmin} I_{\varepsilon,f(u)}$ solves system (\ref{Eu})-(\ref{Eu3}). 
Being independent on $\varepsilon$ the above estimates, together with the continuity assumption in \textbf{(A5')} suffice also to pass to the causal limit $\varepsilon \to 0$ following the scheme presented in Section \ref{causal limit}. 
The case $\varphi^2 \neq 0$ follows from an argument analogous to the one presented in Section \ref{nonconvex energies}.   

\end{proof}
\section{Applications to nonlinear PDEs \label{applications}}

In this section, we shall apply the preceding abstract theory to a couple of concrete nonlinear PDEs.

\subsection{System of  doubly-nonlinear  parabolic equations}

We emphasize again that systems of PDEs may not entail any full  gradient-flow  structure; however, some of them can be reduced to a  nonpotential  perturbation problem for a  (doubly-nonlinear)  gradient flow. The following system of  doubly-nonlinear  differential equations falls within the scope of the abstract theory developed in the present paper (see also \cite{Me}). Let $\Omega$ be a bounded domain of $%
\mathbb{R}
^{d}$ with a smooth boundary $\partial \Omega$ and consider
\begin{alignat}{6}
\alpha_{i}(u_{i}^{\prime})-\Delta_{m}^{a_{i}}u_{i}  &  =g_{i}(u_{1}%
,...,u_{k}) \ &&\text{ in }\Omega\times(0,T]\quad &&\text{ for } \ i=1,...,k\text{,}%
\label{ex 1}\\
\frac{\partial u_{i}}{\partial n}  &  =0 &&\text{ on }\partial\Omega \times (0,T]\quad &&\text{ for } \ i=1,...,k\text{,}\\
u_{i}|_{t=0} &  =u_{0i} &&\text{ in }\Omega \quad &&\text{ for } \ i=1,...,k\text{, }
\label{ex3}%
\end{alignat}
where $n$ denotes the outward unit normal vector on $\partial\Omega$ and  $\alpha_i:%
\mathbb{R}
\rightarrow%
\mathbb{R}
$ are maximal monotone operators  such that there  exist  $p>1$ and a positive
constant $C$ such that
\begin{align*}
&C|s|^{p}-\frac{1}{C}\leq A_{i}(s):=\int_{0}^{s}\alpha_{i}(r)\mathrm{d}r
\\& \text{ and } \ |\alpha_{i}(s)|^{p^{\prime}}\leq C(|s|^{p}+1)
\ \text{ for all } s\in%
\mathbb{R}
\text{, } \ i=1,...,k\text{,}%
\end{align*}
and $\Delta_{m}^{a_{i}}$ is the so-called $m$-Laplace operator with a
coefficient function $a_{i}:\Omega\rightarrow%
\mathbb{R}
$ given by
\[
\Delta_{m}^{a_{i}}v=\nabla\cdot\left(  a_{i}(x)|\nabla v|^{m-2}\nabla
v\right)  \text{, } \quad 1<m<\infty\text{.}%
\]
Here we also assume $u_{0 i }\in W^{1,m}(\Omega)$, $a_{i}\in L^{\infty}(\Omega)$ and
$\bar{a}_{1} \leq a_{i}(x)\leq \bar{a}_{2}$ a.e.~in $\Omega$ for some $\bar{a}_{1},\bar{a}_{2}>0$, for all $i=1,...,k$. Finally, we assume $g_{i}:\mathbb{R}^{k}\rightarrow
\mathbb{R}^{k}$ to be continuous and to satisfy
\[
|g_{i}(u_{1},...,u_{k})|^{p^{\prime}}\leq C\left(  1+\sum_{i=1}^{k}|u_{i}%
|^{p}\right)
\]
for some positive constant $C$. In order to apply the abstract theory, we set
$V=\left(  L^{p}(\Omega)\right)  ^{k}$, $X=\left(  W^{1,m}(\Omega)\right)
^{k}$ and
$$
\psi(u) =\sum_{i=1}^{k}\int_{\Omega}A_{i}(u_{i})\quad \mbox{ for } u \in V.
$$
Note that the functional
\begin{equation}\label{til-phi}
\tilde{\phi}(u)=\frac{1}{m} \sum_{i=1}^{k} \int_{\Omega}a_{i}|\nabla u_{i}|^{m}%
 \quad \mbox{ for all } \ u \in X
 \end{equation}
does not satisfy the coercivity assumption \textbf{(A3)} under the Neumann boundary condition. Thus, we rewrite the
equation (\ref{ex 1}) as%
\[
\alpha_i (u_{i}^{\prime})-\Delta_{m}^{a_{i}}u_{i}+|u_{i}|^{m-2}u_{i}=|u_{i}|^{m-2}%
u_{i}+g_{i}(u_{1},...,u_{k})\text{ in }\Omega\times(0,T]
\]
for $i=1,...,k$, and set
\begin{align*}
\phi(u)  &  =\frac{1}{m}\sum_{i=1}^{k}\int_{\Omega} \left( a_{i}|\nabla u_{i} |^{m}+|u_{i}|^{m} \right)\text{,} \\
f_{i}(u)  &  =g_{i}(u_{1},...,u_{k})+|u_{i}|^{m-2}u_{i} \quad \text{ for all
} \ i=1,...,k\text{.}%
\end{align*}
Then, {\bf (A3)} is satisfied. Moreover, $f:=\left(  f_{1},...,f_{k}\right)$ satisfies assumption
\textbf{(A5)}, provided that $p \geq m$. Alternatively, one may set
\begin{align*}
\varphi^1(u)  &  =\frac{1}{m}\sum_{i=1}^{k}\int_{\Omega} \left( a_{i}|\nabla u_{i} |^{m}+|u_{i}|^{m} \right)\text{,} \quad \varphi^2(u) = \dfrac 1 m \sum_{i=1}^{k}\int_\Omega|u_i|^m,\\
f_{i}(u)  &  =g_{i}(u_{1},...,u_{k}) \quad \text{ for all
} \ i=1,...,k\text{.}%
\end{align*}
 Then, \textbf{(A3)-(A6)} hold true, provided that $p \geq m$.

 \begin{remark}
 {\rm
 \begin{enumerate}
 \item Let us remark that the  nonpotential  term $f$ is needed even in the case $k=1$ and $g=0$ in order to couple the equation with Neumann boundary conditions.  Of course, in this case, one can also overcome the difficulty by introducing a nonconvex energy instead of the  nonpotential  perturbation.
  \item As for Dirichlet problems, thanks to Poincar\'e's inequality, one can check \textbf{(A3)} with $\phi = \tilde \phi$ for any $1 < p < \infty$.
\end{enumerate}
 }
\end{remark}

We refer  the reader  to Section 7 of \cite{Ak-St2} for checking that assumptions \textbf{(A1)-(A4)} are satisfied. Thus, by applying
Theorem \ref{thm convex}, we prove the following:

\begin{theorem}
Let $1 < m \leq p < \infty$ and let the above assumptions be satisfied. Then, for every $\varepsilon > 0$
sufficiently small there exists {\rm (}at least{\rm )} a solution  $\{u_{i\varepsilon}\}_{i}$  to the  elliptic-in-time regularized equation %
\begin{alignat*}{6}
-\varepsilon\left(  \alpha_{i}(u_{i}^{\prime})\right)  ^{\prime}+\alpha
_{i}(u_{i}^{\prime})-\Delta_{m}^{a_{i}}u_{i}  &  =g_{i}(u_{1},...,u_{k}) \ &&\text{
in }\Omega\times(0,T] \ &&\text{ for }i=1,...,k\text{,}\\
\frac{\partial u_{i}}{\partial n}  &  =0&&\text{ on }\partial\Omega \times (0,T] \ &&\text{ for
}i=1,...,k\text{,}\\
u_{i}(0)  &  =u_{0i}&&\text{ in }\Omega&&\text{ for }i=1,...,k\text{,}\\
\varepsilon\alpha_{i}(u_{i}^{\prime}(T))  &  =0 &&\text{ in }\Omega&&\text{ for }i=1,...,k\text{.}%
\end{alignat*}
Moreover, $u_{i\varepsilon}\rightarrow u_{i}$ strongly in $C([0,T];L^{p}%
(\Omega))$ for each $i = 1,2,\ldots,k$,  and the limit $\{u_{i}\}_{i}$  solves system \emph{(\ref{ex 1}%
)-(\ref{ex3})}. As for the Dirichlet problem {\rm (}namely, the Neumann boundary condition is 
replaced by $u_i|_{\partial \Omega} = 0${\rm )}, all the assertions above hold true for all $1 < m,p < \infty$.
\end{theorem}

\begin{remark}\label{R:triplet}
 {\rm
Nonlinear equations such as \eqref{ex 1}-\eqref{ex3} are also treated only for $p \geq 2$ in~\cite{Ak}, where a perturbation theory for  doubly-nonlinear  abstract equation is developed in a framework based on the Gel'fand triplet, $V \hookrightarrow H \equiv H^* \subset V^*$. Indeed, the assumption $p \geq 2$ stems from the triplet, and it cannot be removed in the framework. On the other hand, our abstract theory is developed without assuming  the existence of such a  triplet, 
and therefore, the case $1 < p < 2$ also falls within the  scopes of the  preceding abstract theory.
 } 
\end{remark}

\subsection{Biharmonic equation}

The  abstract  theory developed in the current paper can be  also  applied to the quadratic dissipation potential $\psi(u)=\frac{1}{2}|u|_{V}^{2}$ in a Hilbert space $V$, and then, (\ref{target equation}) reads
\begin{equation}
u^{\prime}+\partial\phi(u)=f(u)\text{.} \label{gradient flows}%
\end{equation}
The WED approach to  nonpotential  perturbation  problems (\ref{gradient flows}) has been developed in \cite{Me}, where equation (\ref{gradient flows}) is formulated in a (single) Hilbert space setting. On the other hand, the following example may not fall within the scope of the theory of~\cite{Me}:  %
\begin{alignat}{4}
u^{\prime}+\left(\Delta\right)^{2}u  &  =\beta\cdot\nabla u \quad &&\text{ in
}\Omega\times(0,T],\label{ex2.1}\\
u  &  =0 &&\text{ on }\partial\Omega \times (0,T],\\
\frac{\partial u}{\partial n}  &  =0 &&\text{ on }\partial\Omega \times (0,T],\\
u(0)  &  =u_{0} &&\text{ in }\Omega\text{,} \label{ex2.2}%
\end{alignat}
where $n$ denotes the outward normal on $\partial\Omega$. Indeed, the  nonpotential  term $f(u) := \beta \cdot \nabla u$ in the right-hand side is not well-defined on the whole of the Hilbert space $H = L^2(\Omega)$. However, this obstacle can be overcome in the current setting. Indeed, by following the approach presented in Section \ref{second approach}, we just require $f$ to be defined over the effective domain of the energy potential $X=D(\varphi^1)$ (cf. \textbf{(A5')}).
 To apply our theory, we assume  that $\Omega$ is a bounded subset of $\mathbb{R}^{d}$ with sufficiently smooth boundary $\partial \Omega$, $u_{0}\in H^{2}(\Omega)\cap H_{0}^{1}(\Omega)$ with $\frac{\partial u_{0}}{\partial n}=0$ on $\partial\Omega$, and $\beta\in L^{\infty}(\Omega,\mathbb{R}^{d})$.  Moreover,  we set two spaces $V=L^{2}(\Omega)$, $X=  H^{2}(\Omega)\cap H_{0}^{1}(\Omega) $ and the energy functional
\[
\phi(u)=\left\{
\begin{array}
[c]{cl}%
\int_{\Omega}\frac{|\Delta u|^{2}}{2} & \text{if } \ u\in H^{2}(\Omega)\cap
H_{0}^{1}(\Omega)\text{ and }\frac{\partial u}{\partial n}=0\text{ on
}\partial\Omega\text{,}\\
\infty & \text{otherwise.}%
\end{array}
\right.
\]
 Furthermore,  set $\psi(u)=\frac{1}{2}|u|_{V}^{2}$, and $f: X \rightarrow 
V^{\ast}$ defined by $f(u)=\beta\cdot\nabla u$ and $p=m=2$. Note that the map $f$ satisfies assumption \textbf{(A5')}.
Indeed it is straightforward to check the growth condition (\ref{growth condition 2}). Furthermore, thanks to the compact embeddings $X \hookrightarrow \hookrightarrow H^1(\Omega) \hookrightarrow \hookrightarrow V$ and to Aubin-Lions-Simon's compactness lemma (see \cite{Si}), we have that the space $L^2(0,T;X) \cap H^1(0,T;V)$ is compactly embedded into $L^2(0,T;H^1(\Omega))$. Thus, for all $u_n \to u $ weakly in $L^2(0,T;X) \cap H^1(0,T;V)$ there exists a (not relabeled) subsequence such that $\nabla u_n \to \nabla u$ strongly in $L^2(0,T;L^2(\Omega))$. Recalling that $V=L^2(\Omega)=V^\ast$ we then have that $f(u_n)=\beta \cdot \nabla u_n \to \beta \cdot \nabla u =f(u)$ strongly in $L^2(0,T;V^\ast)$. Finally, by uniqueness of the limit the convergence holds true for the whole sequence. Thus, assumption \textbf{(A5')} is satisfied. By applying the  
preceding  abstract theory, and more precisely Theorem \ref{thm second approach}, we obtain the following result:

\begin{theorem}
 Let the assumptions mentioned above be satisfied .  Then, for every 
$\varepsilon > 0$ sufficiently small, there exists {\rm (}at least{\rm )}  one  solution $u_{\varepsilon}$ to equation,%
\begin{alignat*}{4}
-\varepsilon u^{\prime\prime}+u^{\prime}+\left(  \Delta\right)  ^{2}u  &
=\beta\cdot\nabla u \quad && \text{ in }\Omega\times(0,T],\\
u  &  =0 &&\text{ on }\partial\Omega \times (0,T],\\
\frac{\partial u}{\partial n}  &  =0 &&\text{ on }\partial\Omega \times (0,T],\\
u(0)  &  =u_{0} &&\text{ in }\Omega\text{,}\\
\varepsilon u^{\prime}(T)  &  =0 &&\mbox{ in } \Omega\text{.}%
\end{alignat*}
Moreover, $u_{\varepsilon}\rightarrow u$ strongly in $C([0,T]; L^2(\Omega))$
and  the limit  $u$ is a solution of equation \emph{(\ref{ex2.1})-(\ref{ex2.2})}.
\end{theorem}

 Here, we dealt with a linear equation just for simplicity. However, we remark that 
also doubly-nonlinear variants of (\ref{ex2.1})-(\ref{ex2.2}) fall within the
framework of the preceding abstract results.   

\appendix
\section{Appendix}

\subsection{Moreau-Yosida regularization with $p$-modulus duality
mappings\label{yosida regularization}}

In this section, we collect the definition and some properties of the
Moreau-Yosida regularization of convex functionals defined on a Banach space.

Let $V$ be a  strictly convex, reflexive, and separable  Banach space such that
its dual space $V^{\ast}$ is strictly convex. For every
$p\in(1,\infty)$, we define the $p$\textit{-modulus duality mapping}
$F:V\rightarrow V^{\ast}$ by
\[
F(\cdot):=\partial_{V}\left(  \frac{|\cdot|_{V}^{p}}{p}\right)  \text{.}%
\]
Note that, as $|\cdot|_{V}^{p}$ is strictly convex, $F(0)=\{0\}$. Since $V^{\ast
}$ is strictly convex, we have that $F$ is single-valued (see, e.g., \cite{Kien}).
Moreover,
\begin{equation}
\left\langle F(u),u\right\rangle_{ V } =|u|_{V}^{p}=|F(u)|_{V^{\ast}}^{p^{\prime}%
}\text{.}\label{p modulus}%
\end{equation}
Given a maximal monotone  graph  $A$ of $V\times V^{\ast}$, we define the \textit{resolvent} $J_{\lambda}:V\rightarrow D(A)$ (with respect to $F$) by
\begin{equation}\label{def:reso}
J_{\lambda}u=u_{\lambda} \quad  \stackrel{\text{def}}{\Longleftrightarrow}  \quad  F\left(  \frac{u_{\lambda}%
-u}{\lambda}\right)  +A(u_{\lambda})\ni0
\end{equation}
for each $u\in V$ and the \textit{Yosida approximation} $A_{\lambda
}:V\rightarrow V^{\ast}$ (with respect to $F$) by
\[
A_{\lambda}(u):=F\left(  \frac{u-J_{\lambda}u}{\lambda}\right)  \text{.}%
\]
Hence, by (\ref{p modulus}),
\begin{equation}
|A_{\lambda}(u)|_{V^{\ast}}^{p^{\prime}}=\left\vert \frac{u-J_{\lambda}%
u}{\lambda}\right\vert _{V}^{p}.\label{eq yos}%
\end{equation}
Let now $\phi:V\rightarrow\lbrack0,\infty]$ be a proper, lower
semicontinuous, and convex functional. For simplicity, we assume $0 \in D(\phi)$. Define the \textit{Moreau-Yosida
regularization} of $\phi$ by
\begin{equation}
\phi_{\lambda}(u)=\inf_{v\in V}\left\{  \frac{\lambda}{p}\left\vert \frac
{u-v}{\lambda}\right\vert _{V}^{p}+\phi(v)\right\}  \text{ for }u\in V\text{.
}\label{yosida reg}%
\end{equation}
Note that, for every $u\in V$, the subdifferential of the convex functional
\[
v \mapsto \frac{\lambda}{p}\left\vert \frac{u-v}{\lambda}\right\vert
_{V}^{p}+\phi(v)
\]
is the operator
\[
v \mapsto F\left(  \frac{u-v}{\lambda}\right)  +\partial_{V}\phi(v)\text{}%
\]
 from $V$ into $V^*$.  
Then, the infimum in (\ref{yosida reg}) is  achieved  at $J_{\lambda}u$, where 
$J_{\lambda}$ is the resolvent for $\partial_{V}\phi$, thus, the following relation is satisfied
\begin{equation}
F\left(  \frac{J_{\lambda}u-u}{\lambda}\right)  +\partial_{V}\phi(J_{\lambda}%
u)\ni0\text{. }\label{yosida}%
\end{equation}
In particular, the subdifferential of the Moreau-Yosida regularization of
$\phi$ corresponds to the Yosida approximation of $\partial_{V}\phi$.  Hence,
\begin{align*}
\phi_{\lambda}(u) &  =\frac{\lambda}{p}\left\vert \frac{u-J_{\lambda}%
u}{\lambda}\right\vert _{V}^{p}+\phi(J_{\lambda}u)\\
&  =\frac{\lambda}{p}\left\vert F\left(  \frac{u-J_{\lambda}u}{\lambda
}\right)  \right\vert _{V^{\ast}}^{p^{\prime}}+\phi(J_{\lambda}u)\\
&  =\frac{\lambda}{p}\left\vert \partial_{V}\phi(J_{\lambda}u)\right\vert
_{V^{\ast}}^{p^{\prime}}+\phi(J_{\lambda}u)\text{.}%
\end{align*}
Moreover, testing relation (\ref{yosida}) with $J_{\lambda}u$, we see that%
\begin{align*}
\lambda\left\vert \frac{u-J_{\lambda}u}{\lambda}\right\vert _{V}^{p}%
+  \phi (J_{\lambda}(u))  & \leq   \phi(0) +\left\vert \left\langle F\left(
\frac{u-J_{\lambda}u}{\lambda}\right)  ,u\right\rangle _{V}\right\vert \\
&  \leq \phi(0) +\frac{\lambda}{2}\left\vert \frac{u-J_{\lambda}%
u}{\lambda}\right\vert _{V}^{p}+C\lambda\left\vert \frac{u}{\lambda
}\right\vert _{V}^{p}\text{,}%
\end{align*}
which implies%
\begin{equation}
\left\vert \frac{u-J_{\lambda}u}{\lambda}\right\vert _{V}^{p}\leq\frac
{2}{\lambda}\phi(0)+2C\frac{\left\vert u\right\vert _{V}^{p}}{\lambda^{p}%
}\text{.}\label{est yos}%
\end{equation}
Thus, thanks to identity (\ref{eq yos}),
\begin{equation}
\left\vert \partial_{V}\phi_{\lambda}(u)\right\vert _{V^{\ast}}^{p^{\prime}}%
\leq\frac{2}{\lambda}\phi(0)+2C\frac{\left\vert u\right\vert _{V}^{p}}%
{\lambda^{p}}\text{.}\label{p growth yosida}%
\end{equation}
 Hence $J_\lambda :V \to V$ and $\partial_{V}\phi_{\lambda} :V \to V^*$ turn out to be bounded operators (for each $\lambda$ fixed). 

We  are now ready  to prove demicontinuity of $\partial_{V}\phi_{\lambda}$.

\begin{lemma}
\label{demicontinuity yosida}For every fixed $\lambda>0$, $\partial_{V}
\phi_{\lambda}$ is demicontinuous, i.e., for every sequence $u_{n}\rightarrow u$
strongly in $V$,  it holds that  $\partial_{V}\phi_{\lambda}(u_{n})\rightarrow\partial_{V}\phi_{\lambda}(u)$ weakly in $V^{\ast}$.
\end{lemma}

\begin{proof}
Let $u_{n}\rightarrow u$ in $V$. Let $v_{n}=\partial_{V}\phi_{\lambda}(u_{n})$.
Since $F(\left(  J_{\lambda}u_{n}-u_{n}\right)  /\lambda)+v_{n}=0$, it follows that 
\begin{align*}
&  \left\langle \left(  J_{\lambda}u_{n}-u_{n}\right)  -\left(  J_{\lambda
}u_{m}-u_{m}\right)  ,F\left(  \frac{J_{\lambda}u_{n}-u_{n}}{\lambda}\right)
-F\left(  \frac{J_{\lambda}u_{m}-u_{m}}{\lambda}\right)  \right\rangle _{V}\\
&  +\left\langle J_{\lambda}u_{n}-J_{\lambda}u_{m},v_{n}-v_{m}\right\rangle
_{V}\\
&  =\left\langle u_{m}-u_{n},F\left(  \frac{J_{\lambda}u_{n}-u_{n}}{\lambda
}\right)  -F\left(  \frac{J_{\lambda}u_{m}-u_{m}}{\lambda}\right)
\right\rangle _{V}.
\end{align*}
As a consequence of the strong convergence $u_{n}\rightarrow u$ and of the
boundedness $|x-J_{\lambda}x|_{V}^{p}\leq C(\lambda)(1+|x|_{V}^{p})$, we observe that 
\[
\lim_{m,n\rightarrow\infty}\left\langle u_{m}-u_{n},F\left(  \frac{J_{\lambda
}u_{n}-u_{n}}{\lambda}\right)  -F\left(  \frac{J_{\lambda}u_{m}-u_{m}}%
{\lambda}\right)  \right\rangle _{V}=0.
\]
Thus, as $\partial_{V}\phi_{\lambda}$ and $F$ are monotone and $\partial_{V}
\phi_{\lambda}(u)  \in  \partial_{V}\phi(J_{\lambda}u)$, we get
$$
\lim_{m,n\rightarrow\infty}\left\langle \left(  J_{\lambda}u_{n}-u_{n}\right)
-\left(  J_{\lambda}u_{m}-u_{m}\right)  ,F\left(  \frac{J_{\lambda}u_{n}%
-u_{n}}{\lambda}\right)  -F\left(  \frac{J_{\lambda}u_{m}-u_{m}}{\lambda
 }\right)  \right\rangle _{V}   =0
 $$
 and
 $$
\lim_{m,n\rightarrow\infty}\left\langle J_{\lambda}u_{n}-J_{\lambda}%
u_{m},v_{n}-v_{m}\right\rangle _{V}    =0.
$$
 From  estimates (\ref{est yos}) and (\ref{p growth yosida}),
there exists a (not-relabeled) subsequence such that $J_{\lambda}%
u_{n}\rightarrow\tilde{u}$ weakly in $V$, $v_{n}\rightarrow v$ weakly in
$V^{\ast}$ and $F(\left(  J_{\lambda}u_{n}-u_{n}\right)  /\lambda)\rightarrow
w$ weakly in  $V^\ast$  for some $\tilde{u}\in V$ and $v,w\in V^{\ast}$. Then, by
\cite[Lemma 1.3, pp. 42]{Ba},  one can conclude that  $v\in\partial_{V}\phi(\tilde{u})$ and $F(\left( \tilde{u}-u\right)  /\lambda)+v=0$. Thus, $\tilde{u}=J_{\lambda}u$, and hence, $v=\partial_{V}\phi_{\lambda}(u)$. Since the limits are unique, $J_{\lambda}u_{n}\rightarrow J_{\lambda}u$ weakly in $V$ and $\partial_{V} \phi_{\lambda}(u_{n})\rightarrow\partial_{V}\phi_{\lambda}(u)$ weakly in $V^{\ast }$  along  the whole sequences $u_{n}$ and $\partial_{V}\phi_{\lambda}(u_{n})$, respectively. Thus, $\partial_{V}\phi_{\lambda}$  turns out to be  demicontinuous.
\end{proof}

\subsection{Auxiliary theorems}

 For the  reader's convenience,  we  collect  here some known results  which  we used in analysis.

\begin{theorem}
[Gronwall's Lemma]\label{gronwall}Let $\alpha,u\in L^{1}(0,T)$ and $B>0$.
Assume  that 
\begin{equation}
u(t)\leq\alpha(t)+\int_{0}^{t}Bu(s)\mathrm{d}s \quad \text{for a.a. }t\in(0,T).
\label{assumption gw}%
\end{equation}
 Then,   it holds that 
\begin{equation}
u(t)\leq\alpha(t)+\int_{0}^{t}B\alpha(s)\exp(B(t-s))\mathrm{d}s \quad \mbox{ for all } \ t \in [0,T]. \label{gw}%
\end{equation}

\end{theorem}

\begin{proof}
Define $v(t):=\exp\left(  -Bt\right)  \int_{0}^{t}Bu(s)\mathrm{d}s$. Then,
$v\in W^{1,1}(0,T),$ $v(0)=0$, and
\[
v^{\prime}(t)=B\exp\left(  -Bt\right)  \left(  u(t)-\int_{0}^{t}%
Bu(s)\mathrm{d}s\right)  \leq B\exp\left(  -Bt\right)  \alpha(t)\ \text{ for
}a.a.~t\in(0,T).
\]
Thus, integrating over $(0,t)$, we get
\[
\exp\left(  -Bt\right)  \int_{0}^{t}Bu(s)\mathrm{d}s=v(t)\leq\int_{0}^{t}%
B\exp\left(  -Bs\right)  \alpha(s)\mathrm{d}s,
\]
which yields%
\begin{equation}
\int_{0}^{t}Bu(s)\mathrm{d}s\leq\int_{0}^{t}B\exp\left(  B(t-s)\right)
\alpha(s)\mathrm{d}s\text{.} \label{conclusion gw}%
\end{equation}
By substituting (\ref{conclusion gw}) into (\ref{assumption gw}), we get
(\ref{gw}).
\end{proof}

We now give a proof of
 \begin{lemma}\label{L:X}
 Under \emph{\textbf{(A3)}} and \emph{\textbf{(A4)}}, it holds that $D(\varphi^1) = X$. Moreover, $\varphi^1_X$ is continuous in $X$.
\end{lemma}

 \begin{proof}
 One readily observes that $\varphi^1_X$ is proper and convex. We first show the lower semicontinuity of $\varphi^1_X$ in $X$. Let $\lambda \in \mathbb R$ and let $u_n \in [\varphi^1_X \leq \lambda] := \{w \in X \colon \varphi^1_X(w) \leq \lambda\}$ be such that $u_n \to u$ strongly in $X$. Then, it follows that
 $$
 \lambda \geq \liminf_{n \to \infty} \varphi^1_X(u_n)
 = \liminf_{n \to \infty} \varphi^1(u_n)
 \geq \varphi^1(u) = \varphi^1_X(u)
 $$
 by the lower semicontinuity of $\varphi^1$ in $V$ and the continuous embedding $X \hookrightarrow V$. Hence, we have $u \in [\varphi^1_X \leq \lambda]$. Therefore, $\varphi^1_X$ is lower semicontinuous in $X$.

 We next claim that $D(\partial_X \varphi^1_X)$ is closed in $X$. Indeed, let $u_n \in D(\partial_X \varphi^1_X)$ and $u \in X$ be such that $u_n \to u$ strongly in $X$. Then, there exists a sequence $\{\eta_n\}$ in $X^*$ such that $\eta_n \in \partial_X \varphi^1_X (u_n)$, and moreover, \textbf{(A4)} implies
 $$
 |\eta_n|_{X^*} \leq C.
 $$
 Hence, we deduce, up to a (not-relabeled) subsequence, that $\eta_n \to \eta$ weakly in $X^*$. From the demicontinuity of $\partial_X \varphi^1_X$, we obtain $u \in D(\partial_X \varphi^1_X)$. Therefore $D(\partial_X \varphi^1_X)$ is closed.

 Now, we are ready to show that $D(\varphi^1) = X$. First, note that \textbf{(A3)} potentially means $D(\varphi^1) \subset X$. Hence it suffices to show the inverse relation. Let $u \in X$. Then, it holds that
 $$
 F_X(J_\lambda u - u) + \lambda \partial_X \varphi^1_X(J_\lambda u) \ni 0,
 $$
 where $J_\lambda : X \to D(\partial_X \varphi^1_X)$ is the resolvent of $\partial_X \varphi^1_X$ and $F_X$ is the duality pairing between $X$ and $X^*$. Fix $v_0 \in D(\varphi^1_X)$. Test the equation above by $J_\lambda u - v_0$ and use the definition of subdifferential to get 
$$
 |J_\lambda u|_X \leq C.
 $$
 Moreover, test the same equation by $J_\lambda u - u$ and apply (A4) to derive
 $$
 |J_\lambda u - u|_X \leq \lambda \ell_4(|J_\lambda u|_V)^{1/m'} \left( |J_\lambda u|_X^m + 1\right)^{1/m'} \leq C \lambda \to 0,
 $$
 which implies that $u \in D(\partial_X \varphi^1_X)$ by the closedness. Thus we conclude that $X \subset D(\partial_X \varphi^1|_X)$, which also yields $D(\varphi^1) = X$.

 Finally, combining the lower semicontinuity of $\varphi^1_X$ in $X$ and the fact that the interior of $D(\varphi^1_X)$ coincides with $X$, we deduce by~\cite[Proposition 2.2]{Ba} that $\varphi^1_X$ is continuous in $X$.
 \end{proof}

Finally, we recall Schaefer's fixed-point theorem (see, e.g.,~\cite{Ev}) below.

\begin{theorem}
[Schaefer's fixed-point theorem]\label{schaefer} \emph{\textbf{\cite[Theorem
4, Chap. 9]{Ev}}} Let $B$ be a Banach space and  let  $S:B\rightarrow B$ be continuous and compact.  Suppose that  $\{u\in B:u=\alpha S(u)$ for $\alpha \in\lbrack0,1]\}$  is  bounded. Then,$\ S$ has a fixed point.
\end{theorem}


\begin{thebibliography}{99}
 \bibitem{AiziYan} S.~Aizicovici and Q.~Yan,
	 \newblock{\rm Convergence theorems for abstract doubly
	 nonlinear differential equations},
	 \newblock{\rm  PanAmer.~Math.~J.} {\bf 7} (1997), 1--17.%

\bibitem {Ak} G.~Akagi,
	  \newblock{Doubly nonlinear evolution equations with non-monotone perturbations in reflexive Banach spaces},
	 \newblock{ J.~Evol.~Equ.} {\bf 11} (2011), 1--41.

 \bibitem{A-PJ} G.~Akagi,
	 \newblock{On some doubly nonlinear parabolic equations},
	 \newblock{``Current advances in nonlinear analysis and related topics''},  GAKUTO Internat. Ser. Math. Sci.  Appl. Gakko-Tosho {\bf 32} (2010),239--254. 


 \bibitem{AMS} G.~Akagi, S.~Melchionna, and U. Stefanelli,
         \newblock{ Weighted Energy-Dissipation approach to doubly nonlinear problems on the half line},
         \newblock{ J. Evol. Equ.} (2017), to appear.


 \bibitem{GO3} G.~Akagi and M.~\^Otani,
         \newblock{Time-dependent constraint problems arising from
		 macroscopic critical-state models for type-II superconductivity
		 and their approximations},
         \newblock{ Adv. Math. Sci.
	 Appl.} {\bf 14} (2004), 683--712.

  \bibitem{ASc} G.~Akagi and G.~Schimperna,
	  \newblock{Subdifferential calculus and doubly nonlinear evolution equations in $L^p$-spaces with variable exponents},
	  \newblock{J.~Funct.~Anal.} {\bf 267} (2014), 173--213.
	 
 \bibitem {Ak-St}G.~Akagi and U.~Stefanelli,
	  \newblock{A variational principle for gradient flows of nonconvex energies},
	 \newblock{J.~Convex Anal.} {\bf 23} (2016), 53--75.

 \bibitem {Ak-St3}G.~Akagi and U.~Stefanelli,
	  \newblock{Doubly nonlinear evolution equations as convex minimization},
	  \newblock{SIAM J.~Math.~Anal.} {\bf 46} (2014), 1922--1945.

  \bibitem {Ak-St2.5} G.~Akagi and U.~Stefanelli,
	   \newblock{Periodic solutions for doubly nonlinear evolution equations},
	   \newblock{ J.~Differential~Equations} {\bf 251} (2011), 1790--1812.

  \bibitem {Ak-St2}G.~Akagi and U.~Stefanelli,
	  \newblock{Weighted energy-dissipation functionals for doubly nonlinear evolution},
	  \newblock{J.~Funct.~Anal.} {\bf 260} (2011), 2541--2578.
	  
  \bibitem {Ak-St1}G.~Akagi and U.~Stefanelli,
	  \newblock{A variational principle for doubly nonlinear evolution},
	  \newblock{Appl.~Math.~Lett.} {\bf 23} (2010), 1120--1124.

 \bibitem{Arai} T.~Arai,
	 \newblock {On the existence of the solution for 
	 $\partial \varphi(u'(t)) + \partial \psi(u(t)) \ni f(t)$},
	 \newblock{ J.~Fac.~Sci.~Univ.~Tokyo Sect.~IA Math.} {\bf 26}  
	 (1979), 75--96.

 \bibitem {Ba}V.~Barbu,
	  \newblock{\it Nonlinear semigroups and differential equations
	  in Banach spaces},
	  Noordhoff, Leyden (1976).

 \bibitem{Barbu75} V.~Barbu,
	 \newblock{Existence theorems for a class of two point
	 boundary problems},
	 \newblock { J.~Differential~Equations} {\bf 17} (1975), 236--257.
	  
 \bibitem {Br3} H.~Br\'{e}zis,
	 \newblock {\em Operateurs Maximaux Monotones et Semi-Groupes de 
	 Contractions dans les Espaces de Hilbert},   
	  \newblock {\rm Math. Studies} {\rm 5}, North-Holland, Amsterdam/New 
	 York (1973). 

  \bibitem{BoSc04} E.~Bonetti and G.~Schimperna,
	  \newblock{Local existence for Fr\'emond's model of damage in elastic materials},
	  \newblock{Contin.~Mech.~Thermodyn.} {\bf 16} (2004), 319--335. 
	  
 \bibitem{Fr1} G.~Bonfanti, M.~Fr\'emond and F.~Luterotti,
	 \newblock{Global solution to a nonlinear system for irreversible phase changes},
	 \newblock{Adv.~Math.~Sci.~Appl.} {\bf 10} (2000), 1--24.
 \bibitem{Fr2} G.~Bonfanti, M.~Fr\'emond and F.~Luterotti,
	 \newblock{Local solutions to the full model of phase transitions with dissipation},
	 \newblock{Adv.~Math.~Sci.~Appl.} {\bf 11} (2001), 791--810.

	  
 \bibitem{Colli} P.~Colli,
	 \newblock {On some doubly nonlinear evolution equations in
	 Banach spaces},
	 \newblock { Jpn. J.~Ind.~Appl.~Math.} {\bf 9} (1992), 181--203.
 \bibitem{CV} P.~Colli and A.~Visintin,
	 \newblock {On a class of doubly nonlinear evolution
	 equations},
	 \newblock {Comm.~Partial Differential Equations} {\bf 15}
	 (1990), 737--756.
%
 \bibitem {Ev}L.C.~Evans,
	  \newblock{\it Partial Differential Equations},
	  American Mathematical Society, U.S.A. (1998).

 \bibitem {DMa}G.~Dal Maso,
	  \newblock{\it An introduction to $\Gamma$-convergence},
	  Birkh\"auser, Basel (1993).

 \bibitem {Hi} N.~Hirano,
	  \newblock{Existence of periodic solutions for nonlinear evolution equations in Hilbert spaces},
	  \newblock{Proc.~Amer.~Math.~Soc.} {\bf 120} (1994), 185--192.

 \bibitem {Il} T.~Ilmanen,
	  \newblock{Elliptic regularization and partial regularity for motion by mean curvature},
	  \newblock{Mem.~Amer.~Math.~Soc.} {\bf 108} (1994), 520:x+90.

 \bibitem {Ke} N.~Kemnochi,
	  \newblock{Some nonlinear parabolic variational inequalities},
	  \newblock{Israel J.~Math.} {\bf 22} (1975), 304--331.

 \bibitem {Kien} B.T.~Kien,
	  \newblock{The normalized duality mapping and two related characteristic properties of a uniformly convex Banach space},
	  \newblock{ Acta Math.~Vietnam.} {\bf 27} (2002), 53--67.

 \bibitem {Li} J.L.~Lions,
	  \newblock{Sur certaines \'{e}quations paraboliques non lin\'{e}aires},
	  \newblock{Bull.~Soc.~Math.~France} {\bf 93} (1965), 155--175.

 \bibitem {Li-Ma} J.L.~Lions and E.~Magenes,
	  \newblock{\it Problem\`{e}s aux limites non homog\`{e}nes et applications},
	  \newblock{Travaux et Recherches Math\'{e}matiques} { \bf 1}, Dunod, Paris (1968).

 \bibitem {Me} S.~Melchionna,
	  \newblock{A variational principle for nonpotential perturbations of gradient flows of nonconvex energies},
J. Differential Equations, to appear (2017).

  \bibitem{Me2} S.~Melchionna,
	  \newblock{ A variational approach to symmetry, monotonicity, and comparison for doubly-nonlinear
equations}, { submitted, arXiv: 1610.04478 (2016)}.
	  
 \bibitem {Mi-Or} A.~Mielke and M.~Ortiz,
	  \newblock{A class of minimum principles for characterizing the trajectories of dissipative systems},
	  \newblock{ESAIM Control Optim.~Calc.~Var.} {\bf 14} (2008), 494--516.

 \bibitem{MiTh} A.~Mielke and F.~Theil,
	 \newblock{On rate-independent hysteresis models},
	 \newblock{NoDEA Nonlinear Differential Equations Appl.} 
	 {\bf 11} (2004), 151--189.

  \bibitem{MiRo} A.~Mielke and R.~Rossi,
	  \newblock{Existence and uniqueness results for a class of
	  rate-independent hysteresis problems},
	  \newblock{Math.~Models Methods Appl.~Sci.} 
	  {\bf 17} (2007), 81--123.

  \bibitem{MRS} A.~Mielke, R.~Rossi, and G.~Savar\'e,
	  \newblock{Nonsmooth analysis of doubly nonlinear evolution equations},
	  \newblock{ Calc.~Var.~Partial Differential Equations} {\bf 46} (2013), 253--310. 
 
 \bibitem {Mi-St} A.~Mielke and U.~Stefanelli,
	  \newblock{Weighted energy-dissipation functionals for gradient flows},
	  \newblock{ESAIM Control Optim.~Calc.~Var.} {\bf 17} (2011), 52--85.

 \bibitem {Ol} O.A.~Oleinik,
	  \newblock{On a problem of G.~Fichera},
	  \newblock{Dolk.~Akad.~Nauk SSSR} {\bf 157} (1964), 1297--1300.

 \bibitem{Ot1} M.~\^Otani, 
	 \newblock{Nonmonotone perturbations for nonlinear
	 parabolic equations associated with subdifferential operators,
	 Cauchy problems},   
	 \newblock { J.~Differential~Equations} {\bf 46} (1982), 268-299.
 \bibitem{Ot2} M.~\^Otani, 
	 \newblock{Nonmonotone perturbations for nonlinear
	 parabolic equations associated with subdifferential operators,
	 Periodic problems},   
	 \newblock { J.~Differential~Equations} {\bf 54} (1984),
	 248-273.

  \bibitem{RR15} E.~Rocca and  R.~Rossi,
	  \newblock{``Entropic'' solutions to a thermodynamically consistent PDE system for phase transitions and damage},
	  \newblock{SIAM J.~Math.~Anal.} {\bf 47} (2015), 2519--2586.

  
  \bibitem{RSSS} R.~Rossi, G.~Savar\'e, A.~Segatti, and U.~Stefanelli,
	  \newblock{A variational principle for gradient flows in metric spaces},
	  \newblock{C.~R.~Math.~Acad.~Sci.~Paris} {\bf 349} (2011), 1225--1228.
	 
  \bibitem{Roubicek} T.~Roub\'i\v cek, 
	  \newblock {\em Nonlinear partial differential equations with applications},
	  \newblock {Internat. Ser. Numer. Math.} {\bf 
	 153}, Birkh\"auser Verlag, Basel (2005). 

 \bibitem{SSS} G.~Schimperna, A.~Segatti, and U.~Stefanelli,
	 \newblock{Well-posedness and long-time behavior for a class
	 of doubly nonlinear equations},
	 \newblock{Discrete Contin.~Dyn.~Syst.}
	 {\bf 18} (2007), 15--38.
 \bibitem{Segatti} A.~Segatti,
	 \newblock {Global attractor for a class of doubly nonlinear
	 abstract evolution equations},
	 \newblock {Discrete Contin.~Dyn.~Syst.} {\bf 14} (2006),
	 801--820.  
 \bibitem{Senba} T.~Senba,
	 \newblock {On some nonlinear evolution equation},
	 \newblock {Funkcial Ekvac.} {\bf 29}  
	 (1986), 243--257.

  \bibitem{S-ken} K.~Shirakawa,
	 \newblock {Large time behavior for doubly nonlinear systems
	 generated by subdifferentials},
	  \newblock {Adv.~Math.~Sci.~Appl.} {\bf 10}
	  (2000), 417--442.
  \bibitem{SIYK98} K.~Shirakawa, A.~Ito, N.~Yamazaki and N.~Kenmochi,
	  \newblock{Asymptotic stability for evolution equations governed by subdifferentials},
	  \newblock{Recent developments in domain decomposition methods and flow problems},
	  \newblock{ed. H.~Fujita, H.~Koshigoe, M.~Mori, M.~Nakamura, T.~Nishida and T.~Ushijima},
	  \newblock{GAKUTO Internat.~Ser.~Math.~Sci.~Appl.}, \textbf{11} (1998), 287--310.

  \bibitem {Si} J.~Simon,
	  \newblock{Compact sets in the space $L^{p}(0,T;B)$},
	  \newblock{ Ann.~Mat.~Pura.~Appl.~(4)} {\bf 146} (1987), 65--96.

  \bibitem{St08} U.~Stefanelli,
	  \newblock{The Brezis-Ekeland principle for doubly nonlinear equations},
	  \newblock{SIAM J.~Control Optim.} {\bf 47} (2008), 1615--1642.
	   
  \bibitem{Visintin} A.~Visitin, 
	  \newblock{\em Models of phase transitions},
	  \newblock{Progr. Nonlinear Differential Equations Appl.} {\bf 28}, Birkh\"auser Boston, Inc., Boston, MA (1996). 

  \bibitem{YO} Y.~Yamada and M.~\^Otani,
	  \newblock{On the Navier-Stokes equations in noncylindrical domains: an approach by the subdifferential operator theory},
	  \newblock{J.~Fac.~Sci.~Univ.~Tokyo Sect.~IA Math.} {\bf 25} (1978), 185--204.

 \end{thebibliography}
\end{document}